\documentclass{amsart}[12pt]
\usepackage{amssymb}
\usepackage[nobysame]{amsrefs}
\usepackage[sc]{mathpazo}
\usepackage{tikz}
  \usetikzlibrary{math,calc,through,backgrounds,shapes,matrix,angles,positioning}
\usepackage{todonotes,comment}
\usepackage{enumerate}
\usepackage{ifthen}
\usepackage{xargs}
\usepackage{bbm}
\usepackage[colorlinks=true,
	urlcolor=blue!20!black,
	citecolor=green!50!black]{hyperref}
\usepackage{cleveref}
\newtheorem{theorem}{Theorem}[section]
\newtheorem{conjecture}[theorem]{Conjecture}
\newtheorem{proposition}[theorem]{Proposition}
\newtheorem{lemma}[theorem]{Lemma}
\newtheorem{corollary}[theorem]{Corollary}
\theoremstyle{definition}
\newtheorem{definition}[theorem]{Definition}
\theoremstyle{remark}
\newtheorem{remark}[theorem]{Remark}

\newcommand{\defs}{\stackrel{\mathrm{def}}{=}}
\newcommand{\defn}[1]{{\color{green!50!black}\emph{#1}}}
\newcommand{\nc}{N\!C}
\newcommand{\nn}{N\!N}

\newcommand{\pnc}{\mathcal{N\!C}}

\newcommand{\Symmetric}{\mathfrak{S}}
\newcommand{\Braid}{\mathfrak{B}}
\newcommand{\ello}{\ell_{1}}
\newcommand{\ellk}{\ell_{k}}
\newcommand{\leqo}{\leq_{1}}
\newcommand{\leqk}{\leq_{k}}

\newcommand{\cyc}{\mathrm{cyc}}
\newcommand{\id}{\mathrm{id}}
\newcommand{\ran}{\mathrm{Ran}}
\newcommand{\ZetaPol}{\mathcal{Z}}
\newcommand{\Krew}{\mathrm{Krew}}
\newcommand{\Fact}{\mathrm{Fact}}

\newcommand{\GS}{\Symmetric}
\newcommand{\GSo}{\Symmetric^{(1)}}
\newcommand{\Park}{\mathcal{P}}
\newcommand{\seq}{\mathbf{s}}

\newcommandx{\polygon}[7][7=1.25]{

  \node[circle,minimum size=#5cm] (#2) at  #1 {};

  \foreach \t in {1,...,#3} {
    \coordinate (#2\t) at ($#1+(-\t*360/#3+360/#3:#4)$);
  }
  \draw[thin,black,fill=white,opacity=0.3] #1 circle (#4);
  \setcounter{intege}{1}
  \pgfmathsetcounter{intege}{1}
  \foreach \object in {#6}{
    \ifthenelse{\not\equal{\object}{}}{
      \filldraw[black] ($#1+($(-\theintege*360/#3+360/#3:#4)$)$) circle(2pt);
    }{
    }
    \node[inner sep=0pt] at ($#1+($#7*(-\theintege*360/#3+360/#3:#4)$)$) {\small$\object$};
    \pgfmathsetcounter{intege}{\theintege+1}
    \setcounter{intege}{\theintege}
  }
}
\newcommandx{\polygonwhite}[7][7=1.25]{

  \node[circle,minimum size=#5cm] (#2) at  #1 {};

  \foreach \t in {1,...,#3} {
    \coordinate (#2\t) at ($#1+(-\t*360/#3+360/#3:#4)$);
  }
  \draw[thin,black,opacity=.3] #1 circle (#4);
  \setcounter{intege}{1}
  \pgfmathsetcounter{intege}{1}
  \foreach \object in {#6}{
    \ifodd\theintege{
      \filldraw[black] ($#1+($(-\theintege*360/#3+360/#3:#4)$)$) circle(2pt);
    }\else{
       \draw[black,fill=lightgray] ($#1+($(-\theintege*360/#3+360/#3:#4)$)$) circle(2pt);
    }\fi
    \node[inner sep=0pt] at ($#1+($#7*(-\theintege*360/#3+360/#3:#4)$)$) {\small$\object$};
    \pgfmathsetcounter{intege}{\theintege+1}
    \setcounter{intege}{\theintege}
  }
}
\newcommandx{\polygonwhitetwo}[7][7=1.25]{

  \node[circle,minimum size=#5cm] (#2) at  #1 {};

  \foreach \t in {1,...,#3} {
    \coordinate (#2\t) at ($#1+(-\t*360/#3+360/#3:#4)$);
  }
  \draw[thin,black,opacity=.3] #1 circle (#4);
  \setcounter{intege}{1}
  \pgfmathsetcounter{intege}{1}
  \foreach \object in {#6}{
    \ifodd\theintege{
    }\else{
       \draw[black,fill=lightgray] ($#1+($(-\theintege*360/#3+360/#3:#4)$)$) circle(2pt);
    }\fi
    \node[inner sep=0pt] at ($#1+($#7*(-\theintege*360/#3+360/#3:#4)$)$) {\small$\object$};
    \pgfmathsetcounter{intege}{\theintege+1}
    \setcounter{intege}{\theintege}
  }
}
\newcounter{intege}

\newcommand{\ncFour}[6]{
    \begin{tikzpicture}
      \polygon{(0,0)}{a}{4}{.5}{1}{$1$,$2$,$3$,$4$}[1.5];
        \begin{pgfonlayer}{background}
          \draw[fill=gray!50!white] (#1) to (#2);
          \draw[fill=gray!50!white] (#3) to (#4);
          \draw[fill=gray!50!white] (#5) to (#6);
        \end{pgfonlayer}
      \end{tikzpicture}
}

\newcommand{\polyFour}[4]{
  \begin{tikzpicture}
      \polygonwhite{(0,0)}{a}{8}{.5}{1}{$ $,$ $,$ $,$ $,$ $,$ $,$ $,$ $}[1.5];
        \begin{pgfonlayer}{background}
          \draw[fill=gray!50!white,thick] (#1) to (#2);
          \draw[fill=gray!50!white,thick] (#3) to (#4);
        \end{pgfonlayer}
      \end{tikzpicture}
}

\usepackage[all]{xy}
\usepackage[T1]{fontenc}
\crefformat{footnote}{#2\footnotemark[#1]#3}
\crefformat{conjecture}{Conjecture~#2#1#3}
\crefformat{lemma}{Lemma~#2#1#3}

\title[$k$-Indivisible Noncrossing Partitions]{$k$-Indivisible Noncrossing Partitions}
\date{}
\author{Henri M{\"u}hle}
\address{Technische Universit{\"a}t Dresden, Institut f{\"u}r Algebra, Zellescher Weg 12--14, 01069 Dresden, Germany.}
\email{henri.muehle@tu-dresden.de}
\author{Philippe Nadeau}
\address{Universit\'e de Lyon, CNRS , UMR5208, Institut Camille Jordan, F-69622 Villeurbanne, France}
\email{nadeau@math.univ-lyon1.fr}
\author{Nathan Williams}
\address{University of Texas at Dallas.}
\email{nathan.f.williams@gmail.com}


\dedicatory{Dedicated to Christian Krattenthaler on the occasion of his 60th birthday.}

\begin{document}

\begin{abstract}
  For a fixed integer $k$, we consider the set of noncrossing partitions, where both the block sizes and the difference between adjacent elements in a block is $1\bmod k$.  We show that these \emph{$k$-indivisible} noncrossing partitions can be recovered in the setting of subgroups of the symmetric group generated by $(k{+}1)$-cycles, and that the poset of $k$-indivisible noncrossing partitions under refinement order has many beautiful enumerative and structural properties.  We encounter $k$-parking functions and some special Cambrian lattices on the way, and show that a special class of lattice paths constitutes a nonnesting analogue.
\end{abstract}

\maketitle

\section{Introduction}

\subsection{Classical noncrossing partitions}\label{sec:classical} 
For an integer $n\geq 0$, a (classical) \defn{noncrossing partition} of the set $[n{+}1]\defs\{1,2,\ldots,n{+}1\}$ is a set partition whose blocks have pairwise disjoint convex hulls when drawn on a regular $(n{+}1)$-gon with vertices labeled clockwise by $[n{+}1]$ (\Cref{fig:nckrew} illustrates some examples).  

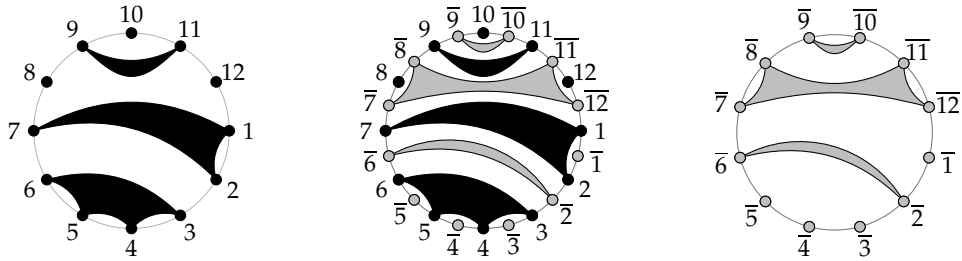
\begin{figure}[htb]
  \begin{tikzpicture}
\polygon{(0,0)}{a}{12}{1.3}{1.4}{$1$,$2$,$3$,$4$,$5$,$6$,$7$,$8$,$9$,$10$,$11$,$12$}[1.2];
      \draw[fill=black] (a1) to[bend right=30] (a2) to[bend right=30] (a7) to[bend left=30] (a1);
      \draw[fill=black] (a3) to[bend right=30] (a4) to[bend right=30] (a5) to[bend right=30] (a6) to[bend left=30] (a3);
      \draw[fill=black] (a9) to[bend right=30] (a11) to[bend left=45, looseness=1.5] (a9);
\end{tikzpicture}\hfill
  \begin{tikzpicture} 
\polygonwhite{(0,0)}{a}{24}{1.3}{1.4}{1,\overline{1},2,\overline{2},3,\overline{3},4,\overline{4},5,\overline{5},6,\overline{6},7,\overline{7},8,\overline{8},9,\overline{9},10,\overline{10},11,\overline{11},12,\overline{12}}[1.2];
      \draw[fill=black] (a1) to[bend right=30] (a3) to[bend right=30] (a13) to[bend left=30] (a1);
      \draw[fill=black] (a5) to[bend right=30] (a7) to[bend right=30] (a9) to[bend right=30] (a11) to[bend left=30] (a5);
      \draw[fill=black] (a17) to[bend right=30] (a21) to[bend left=45, looseness=1.5] (a17);
      \draw[fill=lightgray] (a4) to[bend right=30] (a12) to[bend left=45] (a4);
      \draw[fill=lightgray] (a14) to[bend right=30] (a16) to[bend right=30] (a22)  to[bend right=30] (a24)  to[bend right=15] (a14);
      \draw[fill=lightgray] (a18) to[bend right=30] (a20) to[bend left=45, looseness=1.5] (a18);
      \polygonwhite{(0,0)}{a}{24}{1.3}{1.4}{$ $,$ $,$ $,$ $,$ $,$ $,$ $,$ $,$ $,$ $,$ $,$ $,$ $,$ $,$ $,$ $,$ $,$ $,$ $,$ $,$ $,$ $,$ $,$ $}[1.2];
\end{tikzpicture}\hfill
  \begin{tikzpicture}
\polygonwhitetwo{(0,0)}{a}{24}{1.3}{1.4}{$ $,\overline{1},$ $,\overline{2},$ $,\overline{3},$ $,\overline{4},$ $,\overline{5},$ $,\overline{6},$ $,\overline{7},$ $,\overline{8},$ $,\overline{9},$ $,\overline{10},$ $,\overline{11},$ $,\overline{12}}[1.2];
      \draw[fill=lightgray] (a4) to[bend right=30] (a12) to[bend left=45] (a4);
      \draw[fill=lightgray] (a14) to[bend right=30] (a16) to[bend right=30] (a22)  to[bend right=30] (a24)  to[bend right=15] (a14);
      \draw[fill=lightgray] (a18) to[bend right=30] (a20) to[bend left=45, looseness=1.5] (a18);
      \polygonwhitetwo{(0,0)}{a}{24}{1.3}{1.4}{$ $,$ $,$ $,$ $,$ $,$ $,$ $,$ $,$ $,$ $,$ $,$ $,$ $,$ $,$ $,$ $,$ $,$ $,$ $,$ $,$ $,$ $,$ $,$ $}[1.2];
\end{tikzpicture}

\caption{The leftmost image represents the noncrossing partition $(1\;2\;7)(3\;4\;5\;6)(8)(9\;11)(10)(12)$.  The middle image then illustrates the  computation of its Kreweras complement $(1)(2\;6)(7\;8\;11\; 12)(3)(4)(5) (9\; 10)$, shown on the right.}
\label{fig:nckrew}
\end{figure}

Recall that the symmetric group $\Symmetric_{n{+}1}$ is generated by the set of transpositions $\bigl\{(i\;j)\bigr\}_{1 \leq i<j\leq n{+}1}$, and the noncrossing partitions $\nc_{n{+}1}$ are naturally identified (by sending blocks to cycles) with the elements occurring as prefixes of reduced factorizations of the long cycle $(1\;2\;\ldots\;n{+}1)$ into transpositions~\cite{biane97some}.  Noncrossing partitions lie at the intersection of many seemingly unrelated areas of mathematics---for more information, we refer to the surveys \cites{baumeister19non,mccammond06noncrossing,simion00noncrossing}.

\subsection{$k$-Indivisible noncrossing partitions} Fix integers $k,n \geq 1$.  Throughout this article we write 
\begin{displaymath}
  N\defs kn+1,
\end{displaymath}
and we denote by $\GS_{N;k}$ the subgroup of $\Symmetric_{N}$ generated by the set of all $(k{+}1)$-cycles.

\pagenumbering{arabic}
\addtocounter{page}{1}
\markboth{\SMALL HENRI M\"UHLE, PHILIPPE NADEAU, AND NATHAN WILLIAMS}{\SMALL $k$-INDIVISIBLE NONCROSSING PARTITIONS}

The previous construction of noncrossing partitions as prefixes of reduced factorizations of the long cycle into transpositions naturally generalizes to $\GS_{N;k}$ as the set $\nc_{N;k}$ of elements occurring as prefixes of reduced factorizations of the cycle $c_{N}\defs(1\;2\;\ldots\;N)$ into $(k{+}1)$-cycles.  In reference to Edelman and Armstrong's \defn{$k$-divisible noncrossing partitions} (noncrossing partitions whose block sizes are all divisible by $k$)~\cites{edelman80chain,armstrong09generalized}, we call the elements of $\nc_{N;k}$ the \defn{$k$-indivisible noncrossing partitions}.  Our first result characterizes $\nc_{N;k}$ as a condition on block sizes, explaining the nomenclature ``$k$-indivisible.''

Recall that the \defn{Kreweras complement} of a noncrossing partition $w\in\nc_{N}$ is defined as the coarsest noncrossing partition $\Krew(w)\in\nc_{N}$ that can be drawn on the dual $N$-gon without intersecting $w$ (see~\Cref{fig:nckrew} for an illustration).

\begin{theorem}\label{thm:k_indivisible_equivalent}
  Fix $k,n\geq 1$ and write $N=nk+1$.  The following are equivalent.
  \begin{enumerate}[\rm (i)]
    \item $w$ is a $k$-indivisible noncrossing partition on $[N]$. 
    \item $w$ is a noncrossing partition on $[N]$ and all cycles in both $w$ and its Kreweras complement $\Krew(w)$ have lengths $1\bmod{k}$.
    \item $w$ is a noncrossing partition on $[N]$, all its cycles have lengths $1\bmod{k}$, and if $i<j$ are consecutive in a cycle of $w$, then $j-i\equiv 1\pmod{k}$.
  \end{enumerate}
\end{theorem}
We prove~\Cref{thm:k_indivisible_equivalent} in~\Cref{sec:thm1}.  Note that the $k$-indivisible noncrossing partitions recover the ordinary noncrossing partitions when $k=1$ (so that the congruence constraint on the lengths of blocks is trivially satisfied), and the constructions of~\cite{muehle18poset} when $k=2$.

This combinatorial description allows us to enumerate $\nc_{N;k}$.

\begin{theorem}\label{rem:1modk_bijection}
  The cardinality of $\nc_{N;k}$ is
  \begin{displaymath}
     \frac{2}{N+1}\binom{N+n}{n}.
  \end{displaymath}
\end{theorem}

\subsection{The $k$-indivisible noncrossing partition poset}

As with the classical noncrossing partitions, the set of $k$-indivisible noncrossing partitions is naturally ordered by refinement.   We denote this poset by $\pnc_{N;k}$.  In contrast to when $k=1$, $\pnc_{N;k}$ is generally \emph{not} a lattice.  Nevertheless, we  prove the following formula for its zeta polynomial at the end of~\Cref{sec:1modk_enumeration}.

\begin{theorem}\label{thm:1modk_zeta}
  For $k,n\geq 1$, the number of $q$-multichains of $\pnc_{N;k}$ is 
  \begin{displaymath}
    \ZetaPol_{N;k}(q+1) =\frac{q+1}{Nq+1}\binom{Nq+n}{n}.
  \end{displaymath}
\end{theorem}

The remainder of the paper is devoted to generalizing enumerative results, objects, and bijections from the classical noncrossing partition lattice (obtained by specializing $k$ to $1$) to $\pnc_{N;k}$.

\subsection{$k$-Parking functions}
In~\Cref{sec:1modk_hurwitz}, we give a bijection from the maximal chains of $\pnc_{N;k}$ to $k$-parking functions, generalizing \cite{stanley97parking}*{Theorem~5.1}.

\subsection{Cambrian lattices}
In~\Cref{sec:camb}, we give a bijection from the maximal chains of $\pnc_{N;k}$ up to commutation equivalence, to $(2k{+}2)$-angulations of a convex $2N$-gon following~\cite{mccammond2017noncrossing}.  This construction recovers an instance of a $2k$-Cambrian lattice from \cite{stump18cataland}.

\subsection{Nonnesting partitions}
In~\Cref{sec:1modk_nonnesting} we construct the \defn{$k$-indivisible nonnesting partitions} as the order ideals of a subposet of a triangular poset.  These are shown to be in bijection with the $k$-indivisible noncrossing partitions.

\subsection{Open problems}
We conclude in~\Cref{sec:1modk_mdivisible} with some open problems: 
we conjecture that $\pnc_{N;k}$ is EL-shellable, and we conjecture many enumerative properties of a certain poset whose elements are the $q$-multichains of $\pnc_{N;k}$.

\section{Preliminaries}
  \label{sec:preliminaries}

\subsection{Hurwitz Action}
Let $G$ be a group and let $n\geq 1$.  The $i$-th standard generator $\sigma_i$ of the braid group $\Braid_{n}$ sends $(g_1,g_2,\ldots,g_n)\in G^n$ to 
\begin{displaymath}
  (g_1,g_2,\ldots,g_{i-1},g_{i+1},g_{i+1}^{-1}g_{i}g_{i+1},g_{i+2},\ldots,g_{n})\in G^n.
\end{displaymath}
Its inverse $\sigma_{i}^{-1}$ sends $(g_1,g_2,\ldots,g_n)\in G^n$ to 
\begin{displaymath}
  (g_1,g_2,\ldots,g_{i-1},g_{i}g_{i+1}g_{i}^{-1},g_{i},g_{i+2},\ldots,g_{n})\in G^n.
\end{displaymath}
This is a group action of $\Braid_{n}$ on $G^n$, and it is clear that it does not change the product of such a tuple.  We call this the \defn{Hurwitz action}.

\subsection{$k$-Absolute order}
Let $K\geq 1$ be an integer, and let $\Symmetric_{K}$ be the symmetric group on $[K]$.  For $k\geq 1$ let $C_{K;k}$ be the set of all $(k{+}1)$-cycles of $\Symmetric_{K}$ and let $\GS_{K;k}\leq\Symmetric_{K}$ denote the subgroup generated by $C_{K;k}$.  If $k$ is odd, then $\GS_{K;k}=\Symmetric_K$; if $k$ is even, then $\GS_{K;k}$ is the alternating group $\mathfrak{A}_K$ on $[K]$.

It will be useful to have some notation regarding multiplication by cycles.  Let $(i\; j)$ be a transposition.  If $w\in\Symmetric_{K}$ has two distinct cycles containing $i$ and $j$, we may write $w=w'(\seq_i)(\seq_j)$ where $\seq_i$ and $\seq_j$ are sequences ending with $i$ and $j$, respectively.  Then $w\cdot (i\;j)=w'(\seq_i\; \seq_j)$, and we say that we \defn{join} the two cycles.  More generally, given $m$ disjoint cycles of $w$, we may join them in a new cycle by multiplying by an $m$-cycle having exactly one element in common with each of them.  The inverse operation is called \defn{cutting} a cycle.

Let $\ellk\colon \GS_{K;k}\to\mathbb{N}$ be the map that assigns to $w\in \GS_{K;k}$ the minimum length of a factorization of $w$ into $(k{+}1)$-cycles.  The \defn{$k$-absolute order} is the following partial order on $\GS_{K;k}$:
\begin{displaymath}
  w\leqk w'\quad\text{if and only if}\quad \ellk(w)+\ellk(w^{-1}w') = \ellk(w').
\end{displaymath}

Since the set of $(k{+}1)$-cycles is a full $\Symmetric_K$-conjugacy class, the map $\ellk$ is invariant under $\Symmetric_{K}$-conjugation by \cite{muehle18poset}*{Proposition~2.3}. We are only aware of simple formulas for $\ellk$ for $k\in\{1,2,3\}$. For example, for $k=1$, if we let $\mathrm{cyc}(w)$ denote the number of cycles of $w \in \Symmetric_K$, then $\ello(w)=K-\cyc(w)$. For $k=2$, we have $\ell_2(w)=K-\mathrm{ocyc(w)}$ where $\mathrm{ocyc}(w)$ denotes the number of odd cycles of $w\in\mathfrak{A}_K$ \cite{muehle18poset}. Some general bounds for $\ell_k$ are given in \cite{herzog76representation}.

\subsection{$(1\bmod{k})$-Permutations}

There is a subset of elements of $\GS_{K;k}$ for which $\ellk$ has a similarly simple form.

\begin{definition}
  A permutation $w\in \GS_{K;k}$ is \defn{$1\bmod{k}$} if---when written as a product of disjoint cycles---all cycles of $w$ have length $1\bmod{k}$.  We denote by $\GSo_{K;k}$ the set of all $(1\bmod{k})$-permutations.
\end{definition}

\begin{lemma}\label{lem:good_cycles_length}
A permutation $w \in \Symmetric_K$ is $1 \bmod{k}$ if and only if $\ell_k(w)=\frac{K-\cyc(w)}{k}$.
\end{lemma}
In particular this specializes to the above mentioned well-known fact that $\ello(w)=K-\cyc(w)$ for any permutation $w$.

\begin{proof}
Let $w\in \GS_{K;k}$, and let $t$ be a $(k{+}1)$-cycle. Note that $t$ can be written as a product of $k$ transpositions, and so by analyzing the cut and join possibilities, we obtain that $\cyc(wt)\geq \cyc(w)-k$. Furthermore, equality holds if and only if $t$ has at most one element in common with each cycle of $w$---in this case $wt$ is obtained from $w$ by joining the $k{+}1$ cycles of $w$ that have a common element with $t$.
Now fix a minimal factorization of $w$ into $(k{+}1)$-cycles. By induction, starting from the fact that the identity permutation has $K$ cycles of length $1$, the previous inequality implies that any $w\in \GS_{K;k}$ satisfies $\cyc(w) \geq K-k\ell_k(w)$, and equality occurs if and only if $w$ was built by joining $k{+}1$ cycles at a time, as described above. 

In the case of equality, $w$ is $1\bmod{k}$ since joining $k+1$ cycles of length $1\bmod{k}$ gives back another cycle of length $1\bmod{k}$.
Conversely, every $1\bmod{k}$ permutation can be written as a product of $\frac{K-\cyc(w)}{k}$ elements of $C_{K;k}$, for instance by factoring each of its cycles as follows:
\begin{displaymath}
  (a_1\; a_2\; \ldots\; a_{sk+1}) = (a_1\;\ldots\; a_{k+1})\cdot (a_{k+1}\;\ldots\; a_{2k+1})\cdots (a_{(s-1)k+1}\;\ldots\; a_{sk+1}).\qedhere
\end{displaymath}
\end{proof}

The covering relations $\lessdot_{k}$ of the partial order $\leqk$ in which the top element belongs to $\GSo_{K;k}$ are particularly simple to describe.

\begin{corollary}\label{cor:covering_relations}
Let $w\in \GSo_{K;k}$ and $u\in \GS_{K;k}$. Then one has $u\lessdot_{k} w$ if and only if $u$ can be obtained from $w$ by cutting one cycle of $w$ into $k+1$ cycles of length $1\bmod{k}$.
\end{corollary}
\begin{proof}
This is an immediate corollary of the proof of \Cref{lem:good_cycles_length}.
\end{proof}

\begin{corollary}\label{cor:lower_ideal}
  If $w \in \GSo_{K;k}$ and $u \leqk w$, then $u \in \GSo_{K;k}$, $u^{-1}w\leq_k w$ and $u^{-1} w \in \GSo_{K;k}$.
\end{corollary}
\begin{proof}
That $u \in \GSo_{K;k}$ follows from~\Cref{cor:covering_relations} by induction.  So fix a reduced factorization $w=t_1\cdot t_2 \cdots  t_l$ with $t_i\in C_{K;k}$ for $i\in[l]$ such that $u=t_1t_2\cdots t_s$ for some $s\in[l]$.  Now, $t_{s+1}t_{s+2}\cdots t_l=u^{-1}w$.  The Hurwitz action allows us to write $w=t_{s+1}\cdots t_{l}t'_{1}t'_{2}\cdots t'_{s}$ for certain $t'_i\in C_{K;k}$, so that $u^{-1}w\leq_k w$ as well.
\end{proof}

\section{$k$-Indivisible noncrossing partitions}
  \label{sec:k_indivisible_noncrossing}

\subsection{$k$-Indivisible noncrossing partitions}
For $k,n\geq 1$ and $N=kn+1$, we fix the long cycle $c_N\defs(1\;2\;\ldots\;N).$ Notice that $c_N\in \GSo_{N,k}$, so that $\ellk(c_N)=n$ by \Cref{lem:good_cycles_length}.
\begin{definition}\label{def:k_indivisible}
  The \defn{$k$-indivisible noncrossing partitions} are the elements of 
  \begin{displaymath}
    \nc_{N;k} \defs \{w\in\GS_{N,k}\mid w\leqk c_N\}.
  \end{displaymath}
\end{definition}

We denote the corresponding poset by $\pnc_{N;k}\defs(\nc_{N;k},\leqk)$.  
For $k=1$, the poset $\pnc_{n{+}1;1}=\pnc_{n{+}1}$ is isomorphic to the lattice of noncrossing partitions of $[n{+}1]$~\cite{biane97some}.   \Cref{fig:posetex} illustrates $\pnc_{N;k}$ for $n=3$ and $k=2$. 

\begin{figure}[htbp]
	\centering
	\begin{tikzpicture}
		\def\x{.9};
		\def\y{3.75};
		\def\s{.5};
		\draw(7.7*\x,1.5*\y) node[fill=white,scale=\s,inner sep=.2pt](n1){
			\begin{tikzpicture}
				\polygon{(0,0)}{a}{7}{.5}{.8}{$1$,$2$,$3$,$4$,$5$,$6$,$7$}[1.5];
			\end{tikzpicture}};
		\draw(1*\x,2*\y) node[fill=white,scale=\s,inner sep=.2pt](n7){\begin{tikzpicture}
				\polygon{(0,0)}{a}{7}{.5}{.8}{$1$,$2$,$3$,$4$,$5$,$6$,$7$}[1.5];
				\draw[fill=black] (a1) to[bend left=30] (a6) to[bend right=90] (a7) to[bend right=90] (a1);
			\end{tikzpicture}};
		\draw(2*\x,2*\y) node[fill=white,scale=\s,inner sep=.2pt](n8){\begin{tikzpicture}
				\polygon{(0,0)}{a}{7}{.5}{.8}{$1$,$2$,$3$,$4$,$5$,$6$,$7$}[1.5];
				\draw[fill=black] (a2) to[bend right=90] (a3) to[bend right=10] (a6) to[bend right=10] (a2);
			\end{tikzpicture}};
		\draw(3*\x,2*\y) node[fill=white,scale=\s,inner sep=.2pt](n14){\begin{tikzpicture}
				\polygon{(0,0)}{a}{7}{.5}{.8}{$1$,$2$,$3$,$4$,$5$,$6$,$7$}[1.5];
				\draw[fill=black] (a4) to[bend right=90] (a5) to[bend right=90] (a6) to[bend left=30] (a4);
			\end{tikzpicture}};
		\draw(4*\x,2*\y) node[fill=white,scale=\s,inner sep=.2pt](n4){\begin{tikzpicture}
				\polygon{(0,0)}{a}{7}{.5}{.8}{$1$,$2$,$3$,$4$,$5$,$6$,$7$}[1.5];
				\draw[fill=black] (a1) to[bend right=10] (a4) to[bend right=10] (a7) to[bend right=90] (a1);
			\end{tikzpicture}};
		\draw(5*\x,2*\y) node[fill=white,scale=\s,inner sep=.2pt](n5){\begin{tikzpicture}
				\polygon{(0,0)}{a}{7}{.5}{.8}{$1$,$2$,$3$,$4$,$5$,$6$,$7$}[1.5];
				\draw[fill=black] (a2) to[bend right=90] (a3) to[bend right=90] (a4) to[bend left=30] (a2);
			\end{tikzpicture}};
		\draw(6*\x,2*\y) node[fill=white,scale=\s,inner sep=.2pt](n12){\begin{tikzpicture}
				\polygon{(0,0)}{a}{7}{.5}{.8}{$1$,$2$,$3$,$4$,$5$,$6$,$7$}[1.5];
				\draw[fill=black] (a2) to[bend right=10] (a5) to[bend right=90] (a6) to[bend right=10] (a2);
			\end{tikzpicture}};
		\draw(7*\x,2*\y) node[fill=white,scale=\s,inner sep=.2pt](n2){
			\begin{tikzpicture}
				\polygon{(0,0)}{a}{7}{.5}{.8}{$1$,$2$,$3$,$4$,$5$,$6$,$7$}[1.5];
				\draw[fill=black] (a1) to[bend right=90] (a2) to[bend left=30] (a7) to[bend right=90] (a1);
			\end{tikzpicture}};
		\draw(8*\x,2*\y) node[fill=white,scale=\s,inner sep=.2pt](n6){\begin{tikzpicture}
				\polygon{(0,0)}{a}{7}{.5}{.8}{$1$,$2$,$3$,$4$,$5$,$6$,$7$}[1.5];
				\draw[fill=black] (a7) to[bend right=10] (a3) to[bend right=90] (a4) to[bend right=10] (a7);
			\end{tikzpicture}};
		\draw(9*\x,2*\y) node[fill=white,scale=\s,inner sep=.2pt](n10){\begin{tikzpicture}
				\polygon{(0,0)}{a}{7}{.5}{.8}{$1$,$2$,$3$,$4$,$5$,$6$,$7$}[1.5];
				\draw[fill=black] (a5) to[bend right=90] (a6) to[bend right=90] (a7) to[bend left=30] (a5);
			\end{tikzpicture}};
		\draw(10*\x,2*\y) node[fill=white,scale=\s,inner sep=.2pt](n11){\begin{tikzpicture}
				\polygon{(0,0)}{a}{7}{.5}{.8}{$1$,$2$,$3$,$4$,$5$,$6$,$7$}[1.5];
				\draw[fill=black] (a1) to[bend right=90] (a2) to[bend right=10] (a5) to[bend right=10] (a1);
			\end{tikzpicture}};
		\draw(11*\x,2*\y) node[fill=white,scale=\s,inner sep=.2pt](n15){\begin{tikzpicture}
				\polygon{(0,0)}{a}{7}{.5}{.8}{$1$,$2$,$3$,$4$,$5$,$6$,$7$}[1.5];
				\draw[fill=black] (a3) to[bend right=90] (a4) to[bend right=90] (a5) to[bend left=30] (a3);
			\end{tikzpicture}};
		\draw(12*\x,2*\y) node[fill=white,scale=\s,inner sep=.2pt](n9){\begin{tikzpicture}
				\polygon{(0,0)}{a}{7}{.5}{.8}{$1$,$2$,$3$,$4$,$5$,$6$,$7$}[1.5];
				\draw[fill=black] (a3) to[bend right=10] (a6) to[bend right=90] (a7) to[bend right=10] (a3);
			\end{tikzpicture}};
		\draw(13*\x,2*\y) node[fill=white,scale=\s,inner sep=.2pt](n3){
			\begin{tikzpicture}
				\polygon{(0,0)}{a}{7}{.5}{.8}{$1$,$2$,$3$,$4$,$5$,$6$,$7$}[1.5];
				\draw[fill=black] (a1) to[bend right=90] (a2) to[bend right=90] (a3) to[bend left=30] (a1);
			\end{tikzpicture}};
		\draw(14*\x,2*\y) node[fill=white,scale=\s,inner sep=.2pt](n13){\begin{tikzpicture}
				\polygon{(0,0)}{a}{7}{.5}{.8}{$1$,$2$,$3$,$4$,$5$,$6$,$7$}[1.5];
				\draw[fill=black] (a1) to[bend right=10] (a4) to[bend right=90] (a5) to[bend right=10] (a1);
			\end{tikzpicture}};
		\draw(1*\x,3*\y) node[fill=white,scale=\s,inner sep=.2pt](n29){\begin{tikzpicture}
				\polygon{(0,0)}{a}{7}{.5}{.8}{$1$,$2$,$3$,$4$,$5$,$6$,$7$}[1.5];
				\draw[fill=black] (a1) to[bend left=30] (a6) to[bend right=90] (a7) to[bend right=90] (a1);
				\draw[fill=black] (a3) to[bend right=90] (a4) to[bend right=90] (a5) to[bend left=30] (a3);
			\end{tikzpicture}};
		\draw(2*\x,3*\y) node[fill=white,scale=\s,inner sep=.2pt](n18){\begin{tikzpicture}
				\polygon{(0,0)}{a}{7}{.5}{.8}{$1$,$2$,$3$,$4$,$5$,$6$,$7$}[1.5];
				\draw[fill=black] (a1) to[bend right=90] (a2) to[bend right=90] (a3) to[bend right=10] (a6) to[bend right=90] (a7) to[bend right=90] (a1);
			\end{tikzpicture}};
		\draw(3*\x,3*\y) node[fill=white,scale=\s,inner sep=.2pt](n24){\begin{tikzpicture}
				\polygon{(0,0)}{a}{7}{.5}{.8}{$1$,$2$,$3$,$4$,$5$,$6$,$7$}[1.5];
				\draw[fill=black] (a1) to[bend right=90] (a2) to[bend right=90] (a3) to[bend left=30] (a1);
				\draw[fill=black] (a4) to[bend right=90] (a5) to[bend right=90] (a6) to[bend left=30] (a4);
			\end{tikzpicture}};
		\draw(4*\x,3*\y) node[fill=white,scale=\s,inner sep=.2pt](n26){\begin{tikzpicture}
				\polygon{(0,0)}{a}{7}{.5}{.8}{$1$,$2$,$3$,$4$,$5$,$6$,$7$}[1.5];
				\draw[fill=black] (a1) to[bend right=10] (a4) to[bend right=90] (a5) to[bend right=90] (a6) to[bend right=90] (a7) to[bend right=90] (a1);
			\end{tikzpicture}};
		\draw(5*\x,3*\y) node[fill=white,scale=\s,inner sep=.2pt](n17){\begin{tikzpicture}
				\polygon{(0,0)}{a}{7}{.5}{.8}{$1$,$2$,$3$,$4$,$5$,$6$,$7$}[1.5];
				\draw[fill=black] (a2) to[bend right=90] (a3) to[bend right=90] (a4) to[bend left=30] (a2);
				\draw[fill=black] (a1) to[bend left=30] (a6) to[bend right=90] (a7) to[bend right=90] (a1);
			\end{tikzpicture}};
		\draw(6*\x,3*\y) node[fill=white,scale=\s,inner sep=.2pt](n27){\begin{tikzpicture}
				\polygon{(0,0)}{a}{7}{.5}{.8}{$1$,$2$,$3$,$4$,$5$,$6$,$7$}[1.5];
				\draw[fill=black] (a2) to[bend right=90] (a3) to[bend right=90] (a4) to[bend right=90] (a5) to[bend right=90] (a6) to[bend right=10] (a2);
			\end{tikzpicture}};
		\draw(7*\x,3*\y) node[fill=white,scale=\s,inner sep=.2pt](n22){\begin{tikzpicture}
				\polygon{(0,0)}{a}{7}{.5}{.8}{$1$,$2$,$3$,$4$,$5$,$6$,$7$}[1.5];
				\draw[fill=black] (a1) to[bend right=90] (a2) to[bend left=30] (a7) to[bend right=90] (a1);
				\draw[fill=black] (a4) to[bend right=90] (a5) to[bend right=90] (a6) to[bend left=30] (a4);
			\end{tikzpicture}};
		\draw(8*\x,3*\y) node[fill=white,scale=\s,inner sep=.2pt](n16){\begin{tikzpicture}
				\polygon{(0,0)}{a}{7}{.5}{.8}{$1$,$2$,$3$,$4$,$5$,$6$,$7$}[1.5];
				\draw[fill=black] (a1) to[bend right=90] (a2) to[bend right=90] (a3) to[bend right=90] (a4) to[bend right=10] (a7) to[bend right=90] (a1);
			\end{tikzpicture}};
		\draw(9*\x,3*\y) node[fill=white,scale=\s,inner sep=.2pt](n20){\begin{tikzpicture}
				\polygon{(0,0)}{a}{7}{.5}{.8}{$1$,$2$,$3$,$4$,$5$,$6$,$7$}[1.5];
				\draw[fill=black] (a2) to[bend right=90] (a3) to[bend right=90] (a4) to[bend left=30] (a2);
				\draw[fill=black] (a5) to[bend right=90] (a6) to[bend right=90] (a7) to[bend left=30] (a5);
			\end{tikzpicture}};
		\draw(10*\x,3*\y) node[fill=white,scale=\s,inner sep=.2pt](n21){\begin{tikzpicture}
				\polygon{(0,0)}{a}{7}{.5}{.8}{$1$,$2$,$3$,$4$,$5$,$6$,$7$}[1.5];
				\draw[fill=black] (a1) to[bend right=90] (a2) to[bend right=10] (a5) to[bend right=90] (a6) to[bend right=90] (a7) to[bend right=90] (a1);
			\end{tikzpicture}};
		\draw(11*\x,3*\y) node[fill=white,scale=\s,inner sep=.2pt](n23){\begin{tikzpicture}
				\polygon{(0,0)}{a}{7}{.5}{.8}{$1$,$2$,$3$,$4$,$5$,$6$,$7$}[1.5];
				\draw[fill=black] (a3) to[bend right=90] (a4) to[bend right=90] (a5) to[bend left=30] (a3);
				\draw[fill=black] (a1) to[bend right=90] (a2) to[bend left=30] (a7) to[bend right=90] (a1);
			\end{tikzpicture}};
		\draw(12*\x,3*\y) node[fill=white,scale=\s,inner sep=.2pt](n28){\begin{tikzpicture}
				\polygon{(0,0)}{a}{7}{.5}{.8}{$1$,$2$,$3$,$4$,$5$,$6$,$7$}[1.5];
				\draw[fill=black] (a3) to[bend right=90] (a4) to[bend right=90] (a5) to[bend right=90] (a6) to[bend right=90] (a7) to[bend right=10] (a3);
			\end{tikzpicture}};
		\draw(13*\x,3*\y) node[fill=white,scale=\s,inner sep=.2pt](n19){\begin{tikzpicture}
				\polygon{(0,0)}{a}{7}{.5}{.8}{$1$,$2$,$3$,$4$,$5$,$6$,$7$}[1.5];
				\draw[fill=black] (a1) to[bend right=90] (a2) to[bend right=90] (a3) to[bend left=30] (a1);
				\draw[fill=black] (a5) to[bend right=90] (a6) to[bend right=90] (a7) to[bend left=30] (a5);
			\end{tikzpicture}};
		\draw(14*\x,3*\y) node[fill=white,scale=\s,inner sep=.2pt](n25){\begin{tikzpicture}
				\polygon{(0,0)}{a}{7}{.5}{.8}{$1$,$2$,$3$,$4$,$5$,$6$,$7$}[1.5];
				\draw[fill=black] (a1) to[bend right=90] (a2) to[bend right=90] (a3) to[bend right=90] (a4) to[bend right=90] (a5) to[bend right=10] (a1);
			\end{tikzpicture}};
		\draw(7.3*\x,3.5*\y) node[fill=white,scale=\s,inner sep=.2pt](n30){\begin{tikzpicture}
				\polygon{(0,0)}{a}{7}{.5}{.8}{$1$,$2$,$3$,$4$,$5$,$6$,$7$}[1.5];
				\draw[fill=black] (a1) to[bend right=90] (a2) to[bend right=90] (a3) to[bend right=90] (a4) to[bend right=90] (a5) to[bend right=90] (a6) to[bend right=90] (a7) to[bend right=90] (a1);
			\end{tikzpicture}};
		\draw(n1) -- (n2);
		\draw(n1) -- (n3);
		\draw(n1) -- (n4);
		\draw(n1) -- (n5);
		\draw(n1) -- (n6);
		\draw(n1) -- (n7);
		\draw(n1) -- (n8);
		\draw(n1) -- (n9);
		\draw(n1) -- (n10);
		\draw(n1) -- (n11);
		\draw(n1) -- (n12);
		\draw(n1) -- (n13);
		\draw(n1) -- (n14);
		\draw(n1) -- (n15);
		\draw(n2) -- (n16);
		\draw(n2) -- (n18);
		\draw(n2) -- (n21);
		\draw(n2) -- (n22);
		\draw(n2) -- (n23);
		\draw(n3) -- (n16);
		\draw(n3) -- (n18);
		\draw(n3) -- (n19);
		\draw(n3) -- (n24);
		\draw(n3) -- (n25);
		\draw(n4) -- (n16);
		\draw(n4) -- (n26);
		\draw(n5) -- (n16);
		\draw(n5) -- (n17);
		\draw(n5) -- (n20);
		\draw(n5) -- (n25);
		\draw(n5) -- (n27);
		\draw(n6) -- (n16);
		\draw(n6) -- (n28);
		\draw(n7) -- (n17);
		\draw(n7) -- (n18);
		\draw(n7) -- (n21);
		\draw(n7) -- (n26);
		\draw(n7) -- (n29);
		\draw(n8) -- (n18);
		\draw(n8) -- (n27);
		\draw(n9) -- (n18);
		\draw(n9) -- (n28);
		\draw(n10) -- (n19);
		\draw(n10) -- (n20);
		\draw(n10) -- (n21);
		\draw(n10) -- (n26);
		\draw(n10) -- (n28);
		\draw(n11) -- (n21);
		\draw(n11) -- (n25);
		\draw(n12) -- (n21);
		\draw(n12) -- (n27);
		\draw(n13) -- (n25);
		\draw(n13) -- (n26);
		\draw(n14) -- (n22);
		\draw(n14) -- (n24);
		\draw(n14) -- (n26);
		\draw(n14) -- (n27);
		\draw(n14) -- (n28);
		\draw(n15) -- (n23);
		\draw(n15) -- (n25);
		\draw(n15) -- (n27);
		\draw(n15) -- (n28);
		\draw(n15) -- (n29);
		\draw(n16) -- (n30);
		\draw(n17) -- (n30);
		\draw(n18) -- (n30);
		\draw(n19) -- (n30);
		\draw(n20) -- (n30);
		\draw(n21) -- (n30);
		\draw(n22) -- (n30);
		\draw(n23) -- (n30);
		\draw(n24) -- (n30);
		\draw(n25) -- (n30);
		\draw(n26) -- (n30);
		\draw(n27) -- (n30);
		\draw(n28) -- (n30);
		\draw(n29) -- (n30);
	\end{tikzpicture}
	\caption{The poset $\pnc_{7;2}$.}
	\label{fig:posetex}
\end{figure}
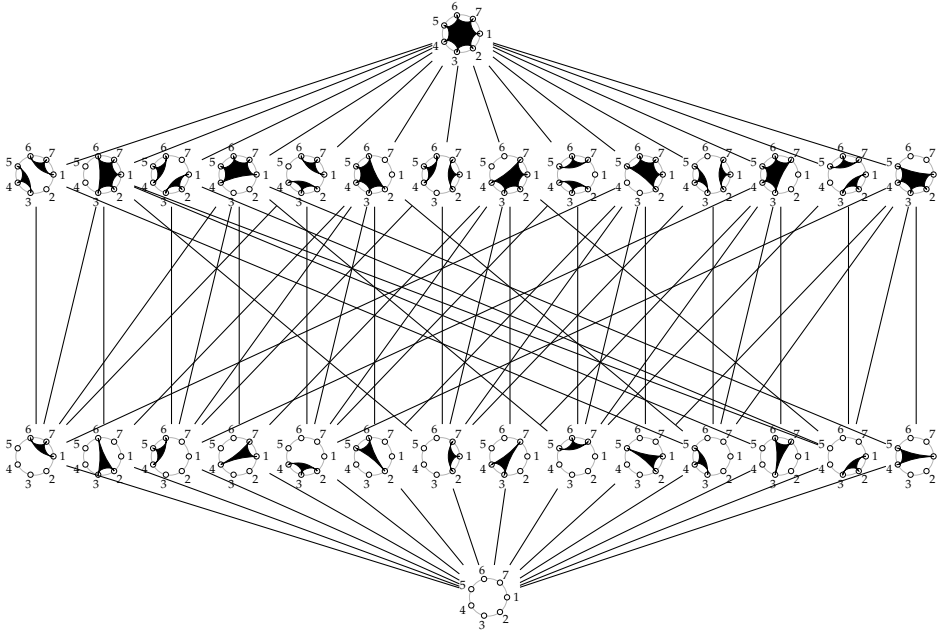

\begin{remark}
Let $\Pi_{K;k}^{(i)}$ be defined as the poset of all partitions of $[K]$ with block sizes congruent to $i \bmod k$.  Some history and results regarding these posets is summarized in~\cite{wachs07poset}*{Examples 4.3.4 and 4.3.5, Exercise 4.3.6, and Remark 4.3.7}---and we are not aware of any substantive results beyond $i=0,1$.

The lattices $\Pi_{K;k}^{(0)}$ first appeared in~\cite{sylvester1976continuous}, and were subsequently studied by Stanley and Sagan in~\cites{stanley1978exponential,sagan1986shellability}.  The corresponding noncrossing partitions were considered by Edelman~\cite{edelman80chain}, and extended to finite Coxeter groups by Armstrong~\cite{armstrong09generalized}.

The posets $\Pi_{K;k}^{(1)}$ were studied in~\cite{calderbank1986partitions}.  However, as far as we know, the corresponding \emph{noncrossing partitions} have not previously been considered.  On the other hand, the study of the maximal chains in $\pnc_{N;k}$ is a classical problem, for example when phrased in the language of transitive factorizations and cacti.  We revisit some of the combinatorics related to these maximal chains in~\Cref{sec:1modk_hurwitz,sec:camb}.
\end{remark}

\subsection{The Kreweras complement}
As in~\Cref{sec:classical}, we graphically represent $w\in\nc_{N}$ as the convex hull of the cycles of $w$ on a regular $N$-gon whose vertices are labeled clockwise by $[N]$.   The terminology ``noncrossing partition'' is justified by the fact that no two convex hulls intersect in this representation.

The \defn{Kreweras complement} of $w\in\nc_{N}$ is the noncrossing partition $\Krew(w)\defs w^{-1} c_N$.  In the graphical representation, this can be visualized by drawing the convex hulls of $w$ on a $2N$-gon labeled clockwise by $\{1,\bar{1},2,\bar{2},\ldots,N,\overline{N}\}$, where the blocks of $w$ use only the non-barred vertices.  Then $\Krew(w)$ corresponds to the coarsest noncrossing partition that can be drawn using the barred vertices without intersecting the blocks of $w$ (see~\Cref{fig:nckrew}).  The following is immediate from~\Cref{cor:lower_ideal}.
\begin{corollary}\label{lem:Kreweras stability}
 For any $w\in \GS_{N;k}$, $w\leq_k c_N$ implies $\Krew(w)\leq_k c_N$; that is, $\nc_{N;k}$ is stable under Kreweras complementation.
\end{corollary}

\subsection{Combinatorial characterization of $k$-indivisible noncrossing partitions}
\label{sec:thm1}

{
\renewcommand{\thetheorem}{\ref{thm:k_indivisible_equivalent}}
\begin{theorem}
Fix $k,n\geq 1$ and write $N=nk+1$.  The following are equivalent.
  \begin{enumerate}[\rm (i)]
    \item $w$ is a $k$-indivisible noncrossing partition on $[N]$. 
    \item $w$ is a noncrossing partition on $[N]$, and $w$ and $\Krew(w)$ are $1\bmod{k}$.
 \item $w$ is a noncrossing partition on $[N]$, $w$ is $1\bmod{k}$, and if $i<j$ are consecutive in a cycle of $w$, then $j-i\equiv 1\pmod{k}$.
  \end{enumerate}
\end{theorem}
\addtocounter{theorem}{-1}
}

Observe that the additional conditions on cycles in $\mathrm{(ii)}$ and $\mathrm{(iii)}$ are vacuous if $k=1$, so the claim is trivial in this case. 

\begin{proof}
   $\mathrm{(i)}\implies\mathrm{(ii)}$.  We assume $w\leq_k c_N$. Since $c_N\in \GSo_{N;k}$, by \Cref{cor:lower_ideal}, we have $w, \Krew(w) \in \GSo_{N;k}$. Now   $\ell_k(w)+\ell_k(\Krew(w))=n$ which can be written as $\ello(w)+\ello(\Krew(w))=nk$ by \Cref{lem:good_cycles_length}. This means that $w\leq_1 c_N$, that is, $w$ is a noncrossing partition.\smallskip

 $\mathrm{(ii)}\implies\mathrm{(iii)}$.
 Let $w\in\nc_{N;1}$ such that both $w$ and $\Krew(w)$ are $1\bmod{k}$.  Let $i,j$ be two consecutive entries in a cycle of $w$ with $i<j$. We want to show that $j-i\equiv1\pmod{k}$. This is trivial if $j=i+1$, and we will assume by induction that this holds for any consecutive entries $i_1<j_1$ in a cycle of $w$ such that $j_1-i_1<j-i$. Consider the maximal (with respect to nesting) cycles of $w$ that are between $i$ and $j$: their number is a multiple of $k$ because this number is one less than the length of a cycle of $\Krew(w)$, which is  $1\bmod{k}$. Order these cycles $\zeta_1,\zeta_2,\ldots,\zeta_{mk}$ so that $\max(\zeta_p)<\min(\zeta_{p+1})$. In fact, if $a_p= \min(\zeta_{p})$ and $b_p=\max(\zeta_p)$, we have $b_p=a_{p+1}-1$ for $p=1,\ldots,mk-1$, with boundary conditions $a_1=i+1$ and $b_{mk}=j-1$. We can therefore write
 \[j-i=\sum_{p=1}^{mk}(b_{p}-a_p)+1+mk.\]
  By induction each $\zeta_p$ satisfies the cycle conditions in $\mathrm{(iii)}$, which immediately implies $b_p-a_p\equiv 0\pmod{k}$. Therefore the expression above for $j-i$ is $1\bmod{k}$ as desired.
  
  $\mathrm{(iii)}\implies\mathrm{(i)}$. Given $w$ a noncrossing partition satisfying the $\bmod k$ conditions of $\mathrm{(iii)}$, we want to prove $w\leqk c_N$. If $w=c_N$ we are done, so we suppose that $w\neq c_N$. We will construct a $w'\in \GS_{N;k}$ such that $w\lessdot_{k}w'$ and $w'$ also satisfies $\mathrm{(iii)}$.
  
  Consider the cycle $\zeta_{0}$ of $w$ containing $1$, $\zeta_0=(u_1<u_2<\cdots<u_{kr_1+1})$ with $u_1=1$.  Since $w\neq c_N$, either there exists $q\in[kr_1]$ such that $u_{q+1}-u_q>1$, or $u_i=i$ for all $i$, in which case pick $q=u_q=kr_1+1<N$ and set $u_{q+1}=N+1$. Now consider the maximal cycles from left to right $\zeta_1,\ldots,\zeta_{d}$ between $u_q$ and $u_{q+1}$, so  $d\geq 1$ by our choice of $q$ . For $1\leq p\leq d$, write $a_p,b_p$ for the minimal and maximal elements of $\zeta_p$, so that we get 
  \[u_{q+1}-u_q=\sum_{p=1}^{d}(b_{p}-a_p)+1+d.\]
Now we have $b_p-a_p\equiv 0\pmod{k}$ as above. Since $u_{q+1}-u_q\equiv 1\pmod{k}$, it follows that $d$ is a multiple of $k$, and so $d\geq k$ because $d\geq 1$.
  Now $\zeta_0\zeta_1\cdots\zeta_{k}\cdot(u_p\;b_1\;\ldots\;b_{k})$ is an increasing cycle that is derived from joining $\zeta_0,\zeta_1,\ldots,\zeta_{k}$. Thus $w'=w\cdot(u_p\;b_1\;\ldots\;b_{k})$ satisfies all conditions in $\mathrm{(iii)}$, so by induction we have $w'\leqk c_N$. Moreover, we have $w\lessdot_{k}w'$ by \Cref{cor:covering_relations},  so that $w\leq_k c_N$.
\end{proof}

\begin{remark}\label{rem:increasing_cycles}
  \Cref{thm:k_indivisible_equivalent} implies that the name ``$k$-indivisible noncrossing partition'' for the elements of $\nc_{N;k}$ is indeed justified: every such element corresponds to a noncrossing partition of [N].  This property is not a priori clear from \Cref{def:k_indivisible}.
  
  As such, each cycle of $w\in\nc_{N;k}$ can be written such that its entries form an increasing sequence of integers.
\end{remark}

\Cref{thm:k_indivisible_equivalent} implies that $\pnc_{N;k}$ is an interval in ${(\GS_{N;k},\leqk)}$.

\begin{corollary}\label{cor:induced_subposet}
  The poset $\pnc_{N;k}$ is an induced subposet of $\pnc_{N;1}$: for all $w,w'\in\nc_{N;k}$, $w\leqk w'$ if and only if $w\leqo w'$.
\end{corollary}
\begin{proof}
  The reader may find it helpful to consider the graphical representation of $\Krew$ given in Figure~\ref{fig:nckrew}. Let $w,w'\in\nc_{N;k}$.  By ~\Cref{thm:k_indivisible_equivalent}, each of $w$, $w'$, $\Krew(w)$, $\Krew(w')$ is $1\bmod{k}$.  
  
  Assume first that $w\leqk w'$, that is, $\ellk(w) + \ellk(w^{-1}w') = \ellk(w')$. Then by \Cref{thm:k_indivisible_equivalent} which applies to all three permutations due to \Cref{cor:lower_ideal}, we get $\frac{\ello(w)}{k} + \frac{\ello(w^{-1}w')}{k} = \frac{\ello(w')}{k}$, which after multiplying by $k$ tells us precisely $w\leqo w'$.

Conversely, assume  $w\leqo w'$. Because $w,w'\in \nc_{N;k}$, this simply means that the supports of the cycles of $w$ are included in those of $w'$. We can thus assume without loss of generality that $w'$ consists of a single cycle. Moreover, because of the invariance of $\ellk$ under conjugation, we can even assume $w'=c_{N'}=(1\;2\;\ldots\; N')$ for $N'=mk+1$ with $m\leq n$. So we have $w\leqo c_N'$ and $w$ is a noncrossing partition on $[N']$. By~\Cref{thm:k_indivisible_equivalent}, using the characterization   $\mathrm{(ii)}$, it follows that $w\leqk c_{N'}$, which achieves the proof. 
\end{proof}

\section{Enumerative properties of $k$-indivisible noncrossing partitions}
  \label{sec:1modk_enumeration}

For integers $n,p,r\geq 1$, let us define the \defn{Raney number} by
\begin{displaymath}
  \ran(n,p,r) \defs \frac{r}{np+r}\binom{np+r}{n}.
\end{displaymath}

The specialization $\ran(n,2,1)$ recovers the Catalan number $\frac{1}{n+1}\binom{2n}{n}$, while $\ran(n,p,1)$ recovers the Fu{\ss}--Catalan number $\frac{1}{(p-1)n+1}\binom{pn}{n}$.   

The Raney numbers satisfy the following Catalan-like recurrence.

\begin{lemma}[\cite{graham94concrete}*{p.~202,~Equation~(5.63)}]\label{lem:raney_recursion}
  For integers $n,p,r,s\geq 1$ we have
  \begin{displaymath}
    \ran(n,p,r+s) = \sum_{i=0}^{n}{\ran(i,p,r)\cdot\ran(n-i,p,s)}.
  \end{displaymath}
\end{lemma}

\begin{remark}\label{rem:tree_bijection}
  Let us say that a plane rooted tree is \defn{$k$-divisible} if each vertex has $0\bmod k$-many children.  It is \defn{$(k{+}1)$-ary} if every non-leaf vertex has exactly $k+1$ children.

  It is well known that $k$-divisible trees with $kn+1$ vertices are enumerated by the Fu{\ss}--Catalan number $\ran(n,k+1,1)$. Such trees $T$ are in bijection with $(k{+}1)$-ary trees $T'$ with $n$ non-leaf vertices. Indeed, start at the root of $T$.  If it has no children, it must be that $n=0$, and we set $T'=T$.  Otherwise, by assumption, the root of $T$ has $ik$ children.  We keep the first $k$ of them, and add a new root child to which we attach all the remaining $(i-1)k$ root children.  We now proceed inductively, until we obtain the desired tree $T'$.  This process is clearly reversible (and thus bijective), by contracting along right-most children.
\end{remark}

\subsection{Cardinality}

{
\renewcommand{\thetheorem}{\ref{rem:1modk_bijection}}
\begin{theorem}
  The cardinality of $\nc_{N;k}$ is
  \begin{displaymath}
    \ran(n,k+1,2) = \frac{2}{N+1}\binom{N+n}{n}.
  \end{displaymath}
\end{theorem}
\addtocounter{theorem}{-1}
}
\begin{proof}
  We will prove this bijectively (see~\Cref{cor:1modk_bijection} for another proof); the reader is invited to look at \Cref{fig:goulden_jackson} which illustrates the bijection.
  
  We first map $w\in\nc_{N;k}$ to the factorization $c_N=w\cdot\Krew(w)$ and apply a classical bijection due to Goulden and Jackson~\cite{goulden92combinatorial}*{Theorem~2.1}.  Since they are reduced, factorizations of the form $w\cdot\Krew(w)$ are in bijection with the set of plane edge-rooted trees with $N$ edges and $N+1$ vertices each of degree $1\bmod k$, with vertices alternately colored white and black.  The white vertices correspond to cycles in $w$, and the black vertices to the cycles in $\Krew(w)$ as follows.  Starting from the root edge 
(moving from white to black), we walk around the tree (keeping the tree to our right).  Each of the $N$ edges of the tree is encountered twice, and we label them by the order in which they are visited when moving from a white to a black vertex.  Reading the cyclic sequence of edge labels clockwise around the white vertices recovers the cycles of $w$; and similarly for the black vertices and $\Krew(w)$.

Break this tree into two by deleting the root edge, and root both of the resulting trees using the vertex adjacent to the deleted root edge.  
Since both $w$ and $\Krew(w)$ are $1\bmod k$, each of the vertices in the resulting pair of trees has a multiple of $k$ many children.
By \Cref{rem:tree_bijection}, the resulting trees are counted by an appropriate Fu{\ss}--Catalan number, from which we conclude that
  \begin{equation}\label{eq:k_indivisible_recurrence}
    \lvert\nc_{N;k}\rvert = \sum_{i=0}^{n}{\ran(i,k+1,1)\cdot\ran(n-i,k+1,1)}.
  \end{equation}
  Hence, $\lvert\nc_{N;k}\rvert$ satisfies the recursion given in \Cref{lem:raney_recursion}, and by checking the initial condition, we see that $\lvert\nc_{N;k}\rvert=\ran(n,k+1,2)$ as desired.
  \end{proof}

\begin{remark}
	Recall from \cite{edelman80chain} that $k$-divisible noncrossing partitions are counted by Fu{\ss}--Catalan numbers, too.  Therefore, \eqref{eq:k_indivisible_recurrence} essentially states that any $k$-indivisible noncrossing partition can be broken into a pair of $k$-divisible noncrossing partitions.
\end{remark}

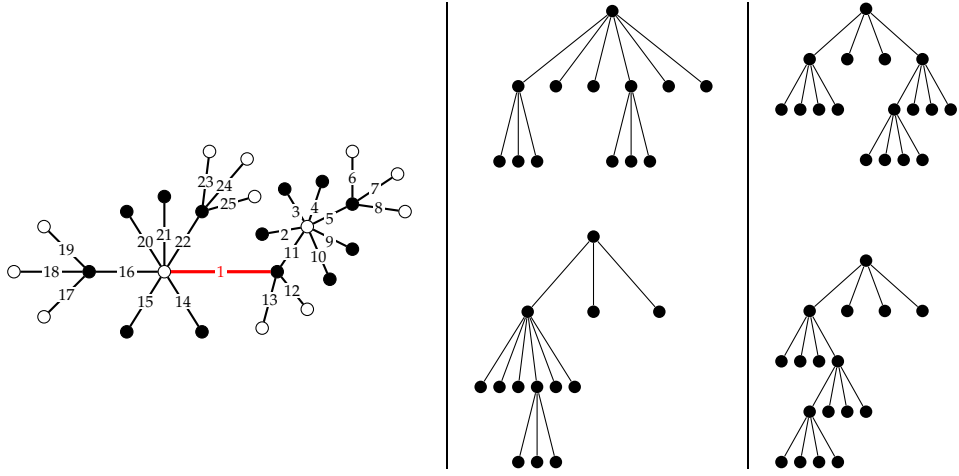
\begin{figure}
  \centering
  \begin{tikzpicture}
    \def\x{1};
    \def\y{1};
    \def\s{.5};
    \draw(1*\x,-.5*\x) node{};
    \draw(1.4*\x,1.4*\y) node[circle,draw,scale=\s](a1){};
    \draw(1*\x,2*\y) node[circle,draw,scale=\s](a2){};
    \draw(1.4*\x,2.6*\y) node[circle,draw,scale=\s](a3){};
    \draw(3*\x,2*\y) node[circle,draw,scale=\s](a4){};
    \draw(3.6*\x,3.6*\y) node[circle,draw,scale=\s](a5){};
    \draw(4.1*\x,3.5*\y) node[circle,draw,scale=\s](a6){};
    \draw(4.2*\x,3*\y) node[circle,draw,scale=\s](a7){};
    \draw(4.3*\x,1.25*\y) node[circle,draw,scale=\s](a8){};
    \draw(4.9*\x,1.5*\y) node[circle,draw,scale=\s](a9){};
    \draw(4.9*\x,2.6*\y) node[circle,draw,scale=\s](a10){};
    \draw(5.5*\x,3.6*\y) node[circle,draw,scale=\s](a11){};
    \draw(6.1*\x,3.3*\y) node[circle,draw,scale=\s](a12){};
    \draw(6.2*\x,2.8*\y) node[circle,draw,scale=\s](a13){};
    \draw(2*\x,2*\y) node[circle,draw,fill,scale=\s](b1){};
    \draw(2.5*\x,1.2*\y) node[circle,draw,fill,scale=\s](b2){};
    \draw(3.5*\x,1.2*\y) node[circle,draw,fill,scale=\s](b3){};
    \draw(2.5*\x,2.8*\y) node[circle,draw,fill,scale=\s](b4){};
    \draw(3*\x,3*\y) node[circle,draw,fill,scale=\s](b5){};
    \draw(3.5*\x,2.8*\y) node[circle,draw,fill,scale=\s](b6){};
    \draw(4.5*\x,2*\y) node[circle,draw,fill,scale=\s](b7){};
    \draw(4.3*\x,2.5*\y) node[circle,draw,fill,scale=\s](b8){};
    \draw(4.6*\x,3.1*\y) node[circle,draw,fill,scale=\s](b9){};
    \draw(5.1*\x,3.2*\y) node[circle,draw,fill,scale=\s](b10){};
    \draw(5.5*\x,2.9*\y) node[circle,draw,fill,scale=\s](b11){};
    \draw(5.5*\x,2.3*\y) node[circle,draw,fill,scale=\s](b12){};
    \draw(5.2*\x,1.9*\y) node[circle,draw,fill,scale=\s](b13){};
    \draw[thick](a1) -- (b1) node[fill=white,inner sep=.6pt] at (1.7*\x,1.7*\y){\tiny $17$};
    \draw[thick](a2) -- (b1) node[fill=white,inner sep=.6pt] at (1.5*\x,2*\y){\tiny $18$};
    \draw[thick](a3) -- (b1) node[fill=white,inner sep=.6pt] at (1.7*\x,2.3*\y){\tiny $19$};
    \draw[thick](a4) -- (b2) node[fill=white,inner sep=.6pt] at (2.75*\x,1.6*\y){\tiny $15$};
    \draw[thick](a4) -- (b3) node[fill=white,inner sep=.6pt] at (3.25*\x,1.6*\y){\tiny $14$};
    \draw[thick](a4) -- (b4) node[fill=white,inner sep=.6pt] at (2.75*\x,2.4*\y){\tiny $20$};
    \draw[thick](a4) -- (b5) node[fill=white,inner sep=.6pt] at (3*\x,2.5*\y){\tiny $21$};
    \draw[thick](a4) -- (b6) node[fill=white,inner sep=.6pt] at (3.25*\x,2.4*\y){\tiny $22$};
    \draw[thick](b1) -- (a4) node[fill=white,inner sep=.6pt] at (2.5*\x,2*\y){\tiny $16$};
    \draw[very thick,red](a4) -- (b7) node[fill=white,inner sep=.6pt] at (3.75*\x,2*\y){\tiny $1$};
    \draw[thick](b6) -- (a5) node[fill=white,inner sep=.6pt] at (3.55*\x,3.2*\y){\tiny $23$};
    \draw[thick](b6) -- (a6) node[fill=white,inner sep=.6pt] at (3.8*\x,3.15*\y){\tiny $24$};
    \draw[thick](b6) -- (a7) node[fill=white,inner sep=.6pt] at (3.85*\x,2.9*\y){\tiny $25$};
    \draw[thick](b7) -- (a8) node[fill=white,inner sep=.6pt] at (4.4*\x,1.625*\y){\tiny $13$};
    \draw[thick](b7) -- (a9) node[fill=white,inner sep=.6pt] at (4.7*\x,1.75*\y){\tiny $12$};
    \draw[thick](b7) -- (a10) node[fill=white,inner sep=.6pt] at (4.7*\x,2.25*\y){\tiny $11$};
    \draw[thick](b8) -- (a10) node[fill=white,inner sep=.6pt] at (4.6*\x,2.5*\y){\tiny $2$};
    \draw[thick](b9) -- (a10) node[fill=white,inner sep=.6pt] at (4.75*\x,2.8*\y){\tiny $3$};
    \draw[thick](b10) -- (a10) node[fill=white,inner sep=.6pt] at (5*\x,2.85*\y){\tiny $4$};
    \draw[thick](b11) -- (a10) node[fill=white,inner sep=.6pt] at (5.2*\x,2.7*\y){\tiny $5$};
    \draw[thick](b12) -- (a10) node[fill=white,inner sep=.6pt] at (5.2*\x,2.4*\y){\tiny $9$};
    \draw[thick](b13) -- (a10) node[fill=white,inner sep=.6pt] at (5.05*\x,2.2*\y){\tiny $10$};
    \draw[thick](b11) -- (a11) node[fill=white,inner sep=.6pt] at (5.5*\x,3.25*\y){\tiny $6$};
    \draw[thick](b11) -- (a12) node[fill=white,inner sep=.6pt] at (5.8*\x,3.1*\y){\tiny $7$};
    \draw[thick](b11) -- (a13) node[fill=white,inner sep=.6pt] at (5.85*\x,2.85*\y){\tiny $8$};
  \end{tikzpicture}
  \hfill
  \vrule
  \hfill
  \begin{tikzpicture}
    \def\x{.5};
    \def\y{1};
    \def\s{.5};
    \draw(3.5*\x,7*\y) node[circle,fill,scale=\s](n1){};
    \draw(1*\x,6*\y) node[circle,fill,scale=\s](n2){};
    \draw(2*\x,6*\y) node[circle,fill,scale=\s](n3){};
    \draw(3*\x,6*\y) node[circle,fill,scale=\s](n4){};
    \draw(4*\x,6*\y) node[circle,fill,scale=\s](n5){};
    \draw(5*\x,6*\y) node[circle,fill,scale=\s](n6){};
    \draw(6*\x,6*\y) node[circle,fill,scale=\s](n7){};
    \draw(.5*\x,5*\y) node[circle,fill,scale=\s](n8){};
    \draw(1*\x,5*\y) node[circle,fill,scale=\s](n9){};
    \draw(1.5*\x,5*\y) node[circle,fill,scale=\s](n10){};
    \draw(3.5*\x,5*\y) node[circle,fill,scale=\s](n11){};
    \draw(4*\x,5*\y) node[circle,fill,scale=\s](n12){};
    \draw(4.5*\x,5*\y) node[circle,fill,scale=\s](n13){};
    \draw(n1) -- (n2);
    \draw(n1) -- (n3);
    \draw(n1) -- (n4);
    \draw(n1) -- (n5);
    \draw(n1) -- (n6);
    \draw(n1) -- (n7);
    \draw(n2) -- (n8);
    \draw(n2) -- (n9);
    \draw(n2) -- (n10);
    \draw(n5) -- (n11);
    \draw(n5) -- (n12);
    \draw(n5) -- (n13);
    \draw(3*\x,4*\y) node[circle,fill,scale=\s](m1){};
    \draw(1.25*\x,3*\y) node[circle,fill,scale=\s](m2){};
    \draw(3*\x,3*\y) node[circle,fill,scale=\s](m3){};
    \draw(4.75*\x,3*\y) node[circle,fill,scale=\s](m4){};
    \draw(0*\x,2*\y) node[circle,fill,scale=\s](m5){};
    \draw(.5*\x,2*\y) node[circle,fill,scale=\s](m6){};
    \draw(1*\x,2*\y) node[circle,fill,scale=\s](m7){};
    \draw(1.5*\x,2*\y) node[circle,fill,scale=\s](m8){};
    \draw(2*\x,2*\y) node[circle,fill,scale=\s](m9){};
    \draw(2.5*\x,2*\y) node[circle,fill,scale=\s](m10){};
    \draw(1*\x,1*\y) node[circle,fill,scale=\s](m11){};
    \draw(1.5*\x,1*\y) node[circle,fill,scale=\s](m12){};
    \draw(2*\x,1*\y) node[circle,fill,scale=\s](m13){};
    \draw(m1) -- (m2);
    \draw(m1) -- (m3);
    \draw(m1) -- (m4);
    \draw(m2) -- (m5);
    \draw(m2) -- (m6);
    \draw(m2) -- (m7);
    \draw(m2) -- (m8);
    \draw(m2) -- (m9);
    \draw(m2) -- (m10);
    \draw(m8) -- (m11);
    \draw(m8) -- (m12);
    \draw(m8) -- (m13);
  \end{tikzpicture}
  \hfill
  \vrule
  \hfill
  \begin{tikzpicture}
    \def\x{.5};
    \def\y{.67};
    \def\s{.5};
    \draw(2.5*\x,10*\y) node[circle,fill,scale=\s](n1){};
    \draw(1*\x,9*\y) node[circle,fill,scale=\s](n2){};
    \draw(2*\x,9*\y) node[circle,fill,scale=\s](n3){};
    \draw(3*\x,9*\y) node[circle,fill,scale=\s](n4){};
    \draw(4*\x,9*\y) node[circle,fill,scale=\s](n5){};
    \draw(.25*\x,8*\y) node[circle,fill,scale=\s](n6){};
    \draw(.75*\x,8*\y) node[circle,fill,scale=\s](n7){};
    \draw(1.25*\x,8*\y) node[circle,fill,scale=\s](n8){};
    \draw(1.75*\x,8*\y) node[circle,fill,scale=\s](n9){};
    \draw(3.25*\x,8*\y) node[circle,fill,scale=\s](n10){};
    \draw(3.75*\x,8*\y) node[circle,fill,scale=\s](n11){};
    \draw(4.25*\x,8*\y) node[circle,fill,scale=\s](n12){};
    \draw(4.75*\x,8*\y) node[circle,fill,scale=\s](n13){};
    \draw(2.5*\x,7*\y) node[circle,fill,scale=\s](n14){};
    \draw(3*\x,7*\y) node[circle,fill,scale=\s](n15){};
    \draw(3.5*\x,7*\y) node[circle,fill,scale=\s](n16){};
    \draw(4*\x,7*\y) node[circle,fill,scale=\s](n17){};
    \draw(n1) -- (n2);
    \draw(n1) -- (n3);
    \draw(n1) -- (n4);
    \draw(n1) -- (n5);
    \draw(n2) -- (n6);
    \draw(n2) -- (n7);
    \draw(n2) -- (n8);
    \draw(n2) -- (n9);
    \draw(n5) -- (n10);
    \draw(n5) -- (n11);
    \draw(n5) -- (n12);
    \draw(n5) -- (n13);
    \draw(n10) -- (n14);
    \draw(n10) -- (n15);
    \draw(n10) -- (n16);
    \draw(n10) -- (n17);
    \draw(2.5*\x,5*\y) node[circle,fill,scale=\s](m1){};
    \draw(1*\x,4*\y) node[circle,fill,scale=\s](m2){};
    \draw(2*\x,4*\y) node[circle,fill,scale=\s](m3){};
    \draw(3*\x,4*\y) node[circle,fill,scale=\s](m4){};
    \draw(4*\x,4*\y) node[circle,fill,scale=\s](m5){};
    \draw(.25*\x,3*\y) node[circle,fill,scale=\s](m6){};
    \draw(.75*\x,3*\y) node[circle,fill,scale=\s](m7){};
    \draw(1.25*\x,3*\y) node[circle,fill,scale=\s](m8){};
    \draw(1.75*\x,3*\y) node[circle,fill,scale=\s](m9){};
    \draw(1*\x,2*\y) node[circle,fill,scale=\s](m10){};
    \draw(1.5*\x,2*\y) node[circle,fill,scale=\s](m11){};
    \draw(2*\x,2*\y) node[circle,fill,scale=\s](m12){};
    \draw(2.5*\x,2*\y) node[circle,fill,scale=\s](m13){};
    \draw(.25*\x,1*\y) node[circle,fill,scale=\s](m14){};
    \draw(.75*\x,1*\y) node[circle,fill,scale=\s](m15){};
    \draw(1.25*\x,1*\y) node[circle,fill,scale=\s](m16){};
    \draw(1.75*\x,1*\y) node[circle,fill,scale=\s](m17){};
    \draw(m1) -- (m2);
    \draw(m1) -- (m3);
    \draw(m1) -- (m4);
    \draw(m1) -- (m5);
    \draw(m2) -- (m6);
    \draw(m2) -- (m7);
    \draw(m2) -- (m8);
    \draw(m2) -- (m9);
    \draw(m9) -- (m10);
    \draw(m9) -- (m11);
    \draw(m9) -- (m12);
    \draw(m9) -- (m13);
    \draw(m10) -- (m14);
    \draw(m10) -- (m15);
    \draw(m10) -- (m16);
    \draw(m10) -- (m17);
\end{tikzpicture}
  \caption{Illustration of the bijection from~\Cref{rem:1modk_bijection} for $n=8$, $k=3$, and $w=(1\;14\;15\;16\;20\;21\;22)(2\;3\;4\;5\;9\;10\;11)\in\nc_{25;3}$.  On the left is the plane, edge-rooted bicolored tree corresponding to $w\cdot\Krew(w)$, in the middle the pair of $3$-divisible trees with a total of $26$ vertices, and on the right the pair of $4$-ary trees with a total of $8$ non-leaf vertices.}
  \label{fig:goulden_jackson}
\end{figure}

\subsection{Multichains}
A \defn{$q$-multichain} in $\nc_{N;k}$ is a tuple $(w_{1},w_{2},\ldots,w_{q})\in\bigl(\GS_{N;k}\bigr)^{q}$ with $w_{1}\leqk w_{2}\leqk\cdots\leqk w_{q}\leqk c_N$.

\begin{lemma}\label{lem:multichain_factorization}
  Each $q$-multichain $(u_{1},u_{2},\ldots,u_{q})$ in $\nc_{N;k}$ corresponds bijectively to a factorization $v_{1}v_{2}\cdots v_{q+1}=c_N$ such that 
  \begin{displaymath}
    \ello(v_{1})+\ello(v_{2})+\cdots+\ello(v_{q+1}) = kn,
  \end{displaymath}
  and $v_{i}\in\GSo_{N;k}$ for $i\in[q+1]$.
\end{lemma}
\begin{proof}
  Let $u_{0}=\id$ and $u_{q+1}=c_N$, and define $v_{i}=u_{i-1}^{-1}u_{i}$ for $i\in[q+1]$.  We immediately see that $v_{1}v_{2}\cdots v_{q+1}=c_N$.  Moreover, since $u_{i}\leqk u_{i+1}$ we conclude from the definition that $\ellk(v_{i+1})=\ellk(u_{i+1})-\ellk(u_{i})$.  We obtain
  \begin{displaymath}
    \sum_{i=1}^{q+1}{\ellk(v_{i})} = \sum_{i=1}^{q+1}{\bigl(\ellk(u_{i})-\ellk(u_{i-1})\bigr)} = \ellk(u_{q+1})-\ellk(u_{0}) = \ellk(c_N)-\ellk(\id) = n.
  \end{displaymath}
  We conclude from \Cref{cor:lower_ideal} that $v_{i}\in\GSo_{N;k}$ and the final claim follows then from \Cref{lem:good_cycles_length}. Conversely, given a factorization $v_{1}v_{2}\cdots v_{q+1}=c_N$, it is easily checked that setting $u_m:=v_1\cdots v_m$ for $m\in [q]$ gives the desired $q$-multichain $(u_{1},u_{2},\ldots,u_{q})$.
\end{proof}

Let $C=(w_{1},w_{2},\ldots,w_{q})$ be a $q$-multichain in $\pnc_{N;k}$, and let $w_{0}=\id$, $w_{q+1}=c_N$.  We define the \defn{rank jump vector} of $C$ by $r(C)\defs(r_{1},r_{2},\ldots,r_{q+1})$, where $r_{i}=\ellk(w_{i})-\ellk(w_{i-1})$ for $i\in[q+1]$.  We write $\ZetaPol_{N;k}(q+1)$ for the number of $q$-multichains of $\pnc_{N;k}$.

\begin{theorem}\label{thm:1modk_rank_enumeration}
  The number of $q$-multichains of $\pnc_{N;k}$ that have the rank jump vector $(r_{1},r_{2},\ldots,r_{q+1})$ is
  \begin{displaymath}
    \frac{1}{N}\prod_{i=1}^{q+1}{\ran(r_{i},1-k,N)} = \frac{1}{N}\prod_{i=1}^{q+1}{\frac{N}{N-(k-1)r_{i}}\binom{N-(k-1)r_{i}}{r_{i}}}.
  \end{displaymath}
\end{theorem}
\begin{proof}
  Let $C=(w_1,w_2,\ldots,w_{q})$ be a $q$-multichain with rank jump vector $r(C)=(r_1,r_2,\ldots,r_{q+1})$, where $r_i=\ellk(w_i)-\ellk(w_{i-1})$.  By \Cref{lem:multichain_factorization}, $C$ corresponds to a factorization $v_1v_2\cdots v_{q+1}=c_N$, where $v_i\in \GSo_{N;k}$ and $r_i=\ellk(v_i)$ for $i\in[q+1]$.  By ~\Cref{lem:good_cycles_length} we have $r_i=\ello(v_i)/k$.  If we suppose that $v_i$ has exactly $p_j^{(i)}$ cycles of size $kj+1$, for $j\geq 1$, then \cite{krattenthaler10decomposition}*{Theorem~5} implies that the number of factorizations is 
  \begin{displaymath}
    N^{q}\prod_{i=1}^{q+1}{\frac{1}{kn-kr_i+1}\binom{kn-kr_i+1}{p_1^{(i)},p_2^{(i)},\dots}},
  \end{displaymath}
  where $r_i = \sum_{j}{jp_{j}^{(i)}}$.  We now sum over all such sequences $(p_1^{(i)},p_2^{(i)},\ldots)$ by using \cite{krattenthaler10decomposition}*{Lemma~4} and find that the number of all such factorizations is
  \begin{displaymath}
    N^{q}\prod_{i=1}^{q+1}{\frac{1}{kn-kr_i+1}\binom{kn-(k-1)r_i}{r_i}} = \frac{1}{N}\prod_{i=1}^{q+1}{\frac{N}{N-kr_i}\binom{N-1-(k-1)r_i}{r_i}}.
  \end{displaymath}
  This formula is equivalent to the formula in the statement.
\end{proof}

\begin{corollary}\label{cor:max_chains}
  The number of maximal chains of $\pnc_{N;k}$ is $N^{n-1}$, and the number of elements of $\pnc_{N;k}$ of rank $l$ is
  \begin{displaymath}
    \frac{N}{\bigl(N-(k-1)l\bigr)\bigl(N-(k-1)(n-l)\bigr)}\binom{N-(k-1)l}{l}\binom{N-(k-1)(n-l)}{n-l}.
  \end{displaymath}
\end{corollary}
\begin{proof}
  Maximal chains of $\pnc_{N;k}$ correspond by definition to $(n-1)$-multichains with rank jump vector $(1,1,\ldots,1)$, while elements of rank $l$ correspond to $1$-multichains with rank jump vector $(l,n-l)$.  The result now follows from \Cref{thm:1modk_rank_enumeration}.
\end{proof}

\begin{remark}
  The result on the number of maximal chains of $\pnc_{N;k}$ has been obtained before by Goulden and Jackson in \cite{goulden1994symmetrical}*{Corollary~5.1}, and was later extended by Biane in \cite{biane96minimal}*{Theorem~1}.
\end{remark}

\subsection{Zeta polynomial and M\"obius function}
We may now conclude \Cref{thm:1modk_zeta}.

{
\renewcommand{\thetheorem}{\ref{thm:1modk_zeta}}
\begin{theorem}
  For $k,n\geq 1$, the number of $q$-multichains of $\pnc_{N;k}$ is 
  \begin{displaymath}
    \ZetaPol_{N;k}(q+1) =\frac{q+1}{Nq+1}\binom{Nq+n}{n}.
  \end{displaymath}
\end{theorem}
\addtocounter{theorem}{-1}
}

\begin{proof}
   In order to determine $\ZetaPol_{N;k}(q+1)$, we have to sum the formula from \Cref{thm:1modk_rank_enumeration} over all possible rank jump vectors.  Recall from \cite{muehle18poset}*{Lemma~5.5} that for integers $a,a_{1},a_{2},\ldots,a_{r},b,n$ with $a=a_{1}+a_{2}+\cdots+a_{r}$ we have
  \begin{displaymath}
    \sum_{n_{1}+n_{2}+\cdots+n_{r}=n}\prod_{i=1}^{r}{\ran(n_{i},b,a_{i})} = \ran(n,b,a).
  \end{displaymath}
  We obtain
  \begin{align*}
    \ZetaPol_{N;k}(q+1) & = \sum_{r_{1}+r_{2}+\cdots+r_{q+1}=n}{\frac{1}{N}\prod_{i=1}^{q+1}{\ran(r_{i},1-k,N)}}\\
    & = \frac{1}{N}\left(\sum_{r_{1}+r_{2}+\cdots+r_{q+1}=n}{\prod_{i=1}^{q+1}{\ran(r_{i},1-k,N)}}\right)\\
    & = \frac{\ran(n,1-k,(q+1)N)}{N}\\
    & = \ran\bigl(n,qk+1,q+1\bigr)\\
    & = \frac{q+1}{Nq+1}\binom{Nq+n}{n}.\qedhere
  \end{align*}
\end{proof}

Specializing \Cref{thm:1modk_zeta} at $q=1$ gives a second (non-bijective) proof of~\Cref{rem:1modk_bijection}.
\begin{corollary}\label{cor:1modk_bijection}
  The cardinality of $\nc_{N;k}$ is $\ran(n,k+1,2)$.
\end{corollary}
\begin{proof}
 Every element of $\nc_{N;k}$ can be regarded as a $1$-multichain of $\pnc_{N;k}$.  The claim thus follows by plugging in $q=1$ into \Cref{thm:1modk_zeta}.
\end{proof}

Since $\pnc_{N;k}$ is a poset with a least and a greatest element, we can define the \defn{M{\"o}bius invariant} of $\pnc_{N;k}$; which is the value $\mu(\pnc_{N;k})$ of the M{\"o}bius function of $\pnc_{N;k}$ applied to $\id$ and $c_N$.  See also \cite{stanley11enumerative_vol1}*{Sections~3.8~and~3.12}.

\begin{corollary}\label{cor:1modk_mobius}
  The M{\"o}bius invariant of $\pnc_{N;k}$ is
  \begin{displaymath}
    \mu(\pnc_{N;k}) = (-1)^{n}\ran(n,2k,1) = \frac{(-1)^{n}}{2nk+1}\binom{2nk+1}{n}.
  \end{displaymath}
\end{corollary}
\begin{proof}
  The numbers $\ZetaPol_{N;k}(q)$ can be regarded as evaluations of a polynomial over the integers.  It follows for instance from \cite{stanley11enumerative_vol1}*{Proposition~3.12.1(c)} that $\mu(\pnc_{N;k})=\ZetaPol_{N;k}(-1)$.  The claim follows from application of \Cref{thm:1modk_zeta}, by using the equality $\binom{-a}{b}=(-1)^{b}\binom{a+b-1}{b}$ for positive integers $a,b$.
\end{proof}

\subsection{$m$-Divisible $k$-indivisible noncrossing partitions}

In the spirit of \cites{armstrong09generalized,edelman80chain} we define a partial order on the set of multichains of $\pnc_{N;k}$.  For an $m$-multichain $C=(x_{1},x_{2},\ldots,x_{m})$ of $\pnc_{N;k}$ we define the \defn{delta sequence} $\delta_{o}(C)=(d_{0};d_{1},\ldots,d_{m})$, where $d_{i}=x_{i}^{-1}x_{i+1}$ for $0\leq i\leq m$, and where we denote by $x_{0}$ the identity and by $x_{m+1}$ the long cycle $c_N$.

For two such multichains $C,C'$ with $\delta_{o}(C)=(d_{0};d_{1},\ldots,d_{m})$ and $\delta_{o}(C')=(d'_{0};d'_{1},\ldots,d'_{m})$ set $C\leq_{\mathbf{k}} C'$ if and only if $d_{i}\geq_k d'_{i}$ for  $1\leq i\leq m$.  Let $\pnc_{N;k}^{(m)}$ denote the corresponding poset.

An earlier version of this article contained Corollaries~\ref{cor:1modk_zeta_mdivisible}--\ref{cor:1modk_mobius_mdivisible} as conjectures.  C.~Krattenthaler has suggested the following generalization of \Cref{thm:1modk_rank_enumeration} to us.

\begin{theorem}[C.~Krattenthaler]\label{thm:1modk_rank_enumeration_mdivisible}
	The number of $q$-multichains of $\pnc_{N;k}^{(m)}$ that have the rank jump vector $(r_{1},r_{2},\ldots,r_{q+1})$ is
	\begin{align*}
		\frac{1}{N} & \ran(r_{1},1-k,N)\prod_{i=2}^{q+1}\ran(r_{i},1-k,mN)\\
		& = \frac{1}{N-(k-1)r_{1}}\binom{N-(k-1)r_{1}}{r_{1}}\prod_{i=2}^{q+1}\frac{mN}{mN-(k-1)r_{i}}\binom{mN-(k-1)r_{i}}{r_{i}}.
	\end{align*}
\end{theorem}
\begin{proof}
	Following \cite{krattenthaler10decomposition}*{Corollary~12}, any such multichain corresponds to a unique factorization 
	\begin{equation}\label{eq:factorizations_mdivisible}
		c_{N} = u_{0}^{(1)}\bigl(v_{1}^{(2)}v_{1}^{(3)}\cdots v_{1}^{(q+1)}\bigr)\bigl(v_{2}^{(2)}v_{2}^{(3)}\cdots v_{2}^{(q+1)}\bigr)\cdots\bigl(v_{m}^{(2)}v_{m}^{(3)}\cdots v_{m}^{(q+1)}\bigr)
	\end{equation}
	into elements from $\pnc_{N;k}^{(m)}$, where
	\begin{displaymath}
		\ell_{1}\bigl(u_{0}^{(1)}\bigr)+\sum_{i,j}\ell_{1}\bigl(v_{j}^{(i)}\bigr) = kn
	\end{displaymath}
	with $\ell_{1}\bigl(u_0^{(1)}\bigr) = kr_{1}$ and
	\begin{displaymath}
		\ell_{1}\bigl(v_{1}^{(i)}\bigr) + \ell_{1}\bigl(v_{2}^{(i)}\bigr) + \cdots + \ell_{1}\bigl(v_{m}^{(i)}\bigr) = kr_{i}
	\end{displaymath}
	for $i\in\{2,3,\ldots,q+1\}$.
	
	First suppose that $\ell_{1}\bigl(v_{j}^{(i)}\bigr)=ks_{j}^{(i)}$, with the $s_{j}^{(i)}$'s fixed, such that
	\begin{equation}\label{eq:length_condition_mdivisible}
		s_{1}^{(i)} + s_{2}^{(i)} + \cdots + s_{m}^{(i)} = r_{i}
	\end{equation}
	for $i\in\{2,3,\ldots,q+1\}$.  The number of factorizations \eqref{eq:factorizations_mdivisible} satisfying \eqref{eq:length_condition_mdivisible} is according to \Cref{thm:1modk_rank_enumeration} precisely
	\begin{equation}\label{eq:fixed_factorizations_mdivisible}
		\frac{1}{N}\ran(r_{1},1-k,N)\prod_{i=2}^{q+1}\prod_{j=1}^{m}\ran\bigl(s_{j}^{(i)},1-k,N\bigr).
	\end{equation}

	The desired number of multichains is now obtained by summing \eqref{eq:fixed_factorizations_mdivisible} over all possible $s_{j}^{(i)}$'s for which \eqref{eq:length_condition_mdivisible} holds.  In view of \Cref{lem:raney_recursion} we conclude that this number is
	\begin{displaymath}
		\frac{1}{N}\ran(r_{1},1-k,N)\prod_{i=2}^{q+1}\ran\bigl(r_{i},1-k,mN\bigr).\qedhere
	\end{displaymath}
\end{proof}

We obtain the following corollaries.

\begin{corollary}\label{cor:1modk_zeta_mdivisible}
  For $k,m,m\geq 1$, the number of $q$-multichains of $\pnc_{N;k}^{(m)}$ is
  \begin{displaymath}
    \ZetaPol_{\pnc_{N;k}^{(m)}}(q+1) = \ran\bigl(n,mkq+1,mq+1\bigr) = \frac{mq+1}{mNq+1}\binom{mNq+n}{n}.
  \end{displaymath}
\end{corollary}
\begin{proof}
	We need to sum the formula from \Cref{thm:1modk_rank_enumeration_mdivisible} over all possible rank jump vectors using \Cref{lem:raney_recursion}, and obtain
	\begin{align*}
		\ZetaPol_{\pnc_{N;k}^{(m)}}(q+1) & = \frac{1}{N}\ran(n,1-k,N+qmN);
	\end{align*}
	this formula is equivalent to the formula in the statement.
\end{proof}

\begin{corollary}\label{cor:1modk_chains_mdivisible}
  For $k,m,n\geq 1$, the number of maximal chains in $\pnc_{N;k}^{(m)}$ is $m^{n}N^{n-1}$.
\end{corollary}
\begin{proof}
	This follows from \Cref{thm:1modk_rank_enumeration_mdivisible} by setting $q=n$ and $r_{1}=0=r_{n+1}$ and $r_{1}=r_{2}=\cdots=r_{n}=1$.
\end{proof}

Observe that $\pnc_{N;k}^{(m)}$ has several minimal elements when $m>1$.  Let $\widehat{\pnc}_{N;k}^{(m)}$ denote the poset that is created from $\pnc_{N;k}^{(m)}$ by adding a least element.  Let $\overline{\pnc}_{N;k}^{(m)}$ denote the poset that is created from $\pnc_{N;k}^{(m)}$ by merging all minimal elements into one.

\begin{corollary}\label{cor:1modk_mobius_mdivisible}
  We have 
  \begin{displaymath}
    \mu\bigl(\widehat{\pnc}_{N;k}^{(m)}\bigr) = (-1)^{n-1}\ran(n,km,m-1) = (-1)^{n-1} \frac{m-1}{Nm-1}\binom{Nm-1}{n},
  \end{displaymath}
  as well as
  \begin{align*}
    \mu\bigl(\overline{\pnc}_{N;k}^{(m)}\bigr) & = (-1)^{n}\Bigl(\ran(n,k(m+1),m) - \ran(n,km,m-1)\Bigr)\\
      & = (-1)^{n} \left(\frac{m}{N(m+1)-1}\binom{N(m+1)-1}{n} - \frac{m-1}{Nm-1}\binom{Nm-1}{n}\right).
  \end{align*}
\end{corollary}
\begin{proof}
	As explained in the proof of \cite{armstrong09generalized}*{Theorem~3.7.7}, it holds that $\mu\bigl(\widehat{\pnc}_{N;k}^{(m)}\bigr) = \ZetaPol_{\pnc_{N;k}^{(m)}}(0)$, and the claimed formula follows from \Cref{cor:1modk_zeta_mdivisible}.  The formula for $\mu\bigl(\overline{\pnc}_{N;k}^{(m)}\bigr)$ follows also from \Cref{cor:1modk_zeta_mdivisible} using an argument verbatim to the one in \cite{armstrong09euler}*{Section~3}.
\end{proof}

For $k=1$, the first equality in \Cref{cor:1modk_mobius_mdivisible} is \cite{armstrong09generalized}*{Theorem~3.7.7}, and the second equality is \cite{armstrong09euler}*{Theorem~3}.

\begin{remark}
	Corollary~12 of \cite{krattenthaler10decomposition} can be used to further refine \Cref{thm:1modk_rank_enumeration_mdivisible} by prescribing the block structure of the first element of such a chain.
\end{remark}

\section{Maximal chains of $\pnc_{N;k}$ and $k$-parking functions}
  \label{sec:1modk_hurwitz}
\subsection{Maximal chains and the Hurwitz action}

Let us denote the set of reduced factorizations of $c_{N}=(1\;2\;\ldots\;N)$ into $(k{+}1)$-cycles by $\Fact_{k}(c_N)$; by construction, these are in bijection with maximal chains in $\pnc_{N;k}$; see also ~\Cref{lem:multichain_factorization}.  Since $C_{N;k}$ is invariant under $\Symmetric_{N}$-conjugation, the Hurwitz action is a bijection on the set of reduced factorizations of $w\in\GS_{N;k}$ into $(k{+}1)$-cycles.

\begin{theorem}\label{thm:snaction}
  For $k,n\geq 1$ the braid group $\Braid_{n}$ acts transitively on $\Fact_{k}(c_N)$.
\end{theorem}
\begin{proof}
  This is a special case of \cite{lando04graphs}*{Theorem~5.4.11}. One may also give a direct inductive proof as in \cite{muehle18poset}*{Proposition 6.2} for the case $k=2$.
\end{proof}

We write the entries of a cycle in a factorization $\mathbf{t}=t_1t_2\cdots t_n\in\Fact_{k}(c_N)$ in increasing order as $t_i=(t_{i,1}<t_{i,2}<\cdots<t_{i,k{+}1})$, which is well defined by \Cref{rem:increasing_cycles}.  A factorization $\mathbf{t}\in \Fact_{k}(c_N)$ is \defn{non-decreasing} if $t_{1,1}\leq t_{2,1}\leq\cdots\leq t_{n,1}$.

\begin{lemma}\label{lem:snaction}
For $k,n\geq 1$ there is an action of the symmetric group $\Symmetric_n$ on $\Fact_k(c_N)$ which restricts to the permutation action on the set of smallest elements of each factor $\{t_{i,1}\}_{i=1}^n$.
\end{lemma}

\begin{proof}
Such an action is known to exist for $k=1$, see~\cites{stanley97parking,biane02parking}; we generalize it here.  Consider the simple transposition $s_i=(i\; i+1)$, and a factorization $\mathbf{t}=t_1t_2\cdots t_n$ in $\Fact_k(c_N)$. The action of $s_i$ on $\mathbf{t}$ is defined as follows: it acts as the Hurwitz operator $\sigma_i$ if $t_{i,1}<t_{i+1,1}$; as the inverse Hurwitz operator $\sigma^{-1}_i$ if $t_{i,1}>t_{i+1,1}$; and as the identity if $t_{i,1}=t_{i+1,1}$. 

One verifies that $s_i$ transposes the values of $t_{i,1}$ and $t_{i+1,1}$: this uses the fact that the product $t_it_{i+1}$ is an increasing cycle. From this, one easily checks that one can extend this to an action of the  symmetric group by showing that the defining relations of $\GS_n$ hold.
\end{proof}

\subsection{$k$-Parking functions} 
We proved in~\Cref{cor:max_chains} that the maximal chains of $\pnc_{N;k}$ are enumerated by $N^{n-1}$.  In this section, we generalize Stanley's bijection in~\cite{stanley97parking} between maximal chains in the noncrossing partition lattice and parking functions.  In recent work, J.~Irving and A.~Rattan found the same generalization of Stanley's bijection.  We thank them for bringing \cites{irving2016parking, irving2019trees} to our attention at CanaDAM 2019.

For $k,n\geq 1$ define a \defn{$k$-parking function} of length $n$ to be any permutation of an integer tuple $(a_{1},a_{2},\ldots,a_{n})$ satisfying $1\leq a_{i}\leq k(i-1)+1$ for $i\in[n]$.  We write $\Park_{N;k}$ for the set of all $k$-parking functions.  We call $(a_1,a_2,\ldots,a_n)\in\Park_{N;k}$ \defn{non-decreasing} if $a_1\leq a_2\leq\cdots\leq a_n$.

It is a routine application of the cycle lemma (and follows from \cite{yan01generalized}*{Theorem~1}) that the number of $k$-parking functions of length $n$ is $N^{n-1}$.  Note also that there is an obvious $\Symmetric_{n}$-action on $\Park_{N;k}$, obtained by permuting the entries.

\begin{theorem}\label{thm:parking_functions}
  For $k,n\geq 1$, the map from maximal chains in $\pnc_{N;k}$ to $k$-parking functions
  \begin{align*}
    \phi\colon\Fact_{k}(c_N) &\to \Park_{N;k} \\
    t_1t_2\cdots t_n &\mapsto (t_{1,1},t_{2,1},\ldots,t_{n,1})
  \end{align*}
  is a bijection.
\end{theorem}
\begin{proof}
  We give a proof based on the $k=1$ case from \cite{biane02parking}.  It is enough to show that $\phi$ is a bijection between non-decreasing factorizations of $c_N$ and non-decreasing $k$-parking functions---indeed the map is clearly equivariant with respect to the symmetric group actions on parking functions and on factorizations from \Cref{lem:snaction}. 
  
  We show by induction on $n$ that, if $t_1t_2\cdots t_n\in\Fact_k(c_N)$ is non-decreasing,
  then $(t_{1,1},t_{2,1},\ldots,t_{n,1})$ is a non-decreasing $k$-parking function.  To prove this, we first claim that if $t_1t_2\cdots t_n\in\Fact_{k}(c_N)$ with $t_{1,1} \leq t_{2,1} \leq \cdots \leq t_{n,1}$, then we must have 
  \begin{displaymath}
    t_{n,1}=t_{n,2}-1=\cdots=t_{n,k{+}1}-k.
  \end{displaymath}
  Since we may write the factorization
  \begin{align*}
    t_1 t_2 \cdots t_{n-1} &= c_N t_n^{-1} \\
    &=(1\;2\; \ldots \;N) (t_{n,k{+}1}\; \ldots\; t_{n,2}\; t_{n,1})\\
    &= (1\;2\;\ldots \;t_{n,1}\; t_{n,k{+}1}{+}1\;\ldots \; N) (t_{n,1}{+}1\;\ldots\;t_{n,2})\cdots (t_{n,k}{+}1\;\ldots\;t_{n,k{+}1}),
  \end{align*}
where the last factorization is into disjoint cycles, each of $t_1,t_2,\ldots,t_{n-1}$ must have support in the set $\{1,2,\ldots,t_{n,1},t_{n,k{+}1}+1,\ldots,N\}$ (by~\cite{biane02parking}*{(F)}).  Therefore, each cycle $(t_{n,i}{+}1\;\ldots\;t_{n,i+1})$ is trivial, from which the claim follows.

By induction, $(t_{1,1},t_{2,1},\ldots,t_{n{-}1,1})$ is a non-decreasing $k$-parking function of length $n-1$.  By assumption we have $t_{n{-}1,1}\leq t_{n,1}$, and since $t_{n,1}+k=t_{n,k{+}1}\leq kn+1$ we conclude $t_{n,1}\leq k(n-1)+1$.  Thus, $(t_{1,1},t_{2,1},\ldots,t_{n,1})$ is a non-decreasing $k$-parking function of length $n$.
\end{proof}

\section{Cambrian Lattices}
\label{sec:camb}

Let $\mathbf{u}=u_1u_2\cdots u_n$ and $\mathbf{v}=v_1v_2\cdots v_n$ be two reduced factorizations of $c_N$ into $(k{+}1)$-cycles.  We say that $\mathbf{u}$ and $\mathbf{v}$ are \defn{commutation equivalent} if $\mathbf{u}$ can be obtained from $\mathbf{v}$ by a sequence of Hurwitz moves on adjacent cycles with disjoint support (so that each move acts as a commutation).

\begin{theorem}[\cite{goulden1994symmetrical}*{Theorem~5.5}]\label{thm:1modk_commutation_equivalent}
  The number of reduced factorizations of $c_{N}$\break into $(k{+}1)$-cycles up to commutation equivalence is the Fu{\ss}--Catalan number\break $\ran(n,2k{+}1,1)$.
\end{theorem}

\begin{remark}
  This result was proven for $k=1$ by Eidswick and Longyear~\cites{eidswick1989short,longyear1989peculiar}, while Springer solved a more general factorization problem in~\cite{springer1996factorizations}.

  More recently, such factorizations for $k=1$ were considered in the context of the associahedron by McCammond~\cite{mccammond2017noncrossing}, which led us to develop the combinatorics of this section.
\end{remark}

There is another well-known set with this same cardinality.

\begin{theorem}[\cite{fuss91solutio}]\label{thm:1modk_polygon_dissections}
  The number of $(2k{+}2)$-angulations of a convex $2N$-gon is given by ${\ran(n,2k{+}1,1)}$.
\end{theorem}

Following \cite{mccammond2017noncrossing}*{Section~3}, we now describe a bijection between the objects of \Cref{thm:1modk_polygon_dissections} and ~\Cref{thm:1modk_commutation_equivalent}.

\begin{theorem}\label{thm:1modk_factorizations_polygon_dissections}
  For $k,n\geq 1$, there is a bijection $\Theta$ between the commutation equivalence classes  of reduced factorizations of $(1\;2\;\ldots\;N)$ into $(k{+}1)$-cycles, and the set of $(2k{+}2)$-angulations of a convex $2N$-gon.
\end{theorem}
\begin{proof}
  Let $\mathbf{t}=t_1t_2\cdots t_n\in\Fact_k(c_N)$.  We visualize $\mathbf{t}$ by drawing the convex hulls of the factors $t_1$, $t_2$, $\ldots$, $t_n$ on a convex polygon with $N$ labeled vertices.  Since $\mathbf{t}$ is a minimal factorization of $c_N$, these convex hulls intersect pairwise in at most one vertex, and every vertex is contained in at least one convex hull.  If we were to label these hulls with the order in which the factor appeared this would be a bijection---forgoing these labels records only the commutation class of the factorization: for every vertex at which at least two convex hulls meet, we can determine the order of the corresponding factors by taking the order counterclockwise around the vertex inside the polygon.  This produces a partial order on the convex hulls, every linear extension of this partial order is a reduced factorization of $c_N$, and any two linear extensions differ only by a commutation of letters.

  We now perform a procedure very similar to the Kreweras complement on these unlabeled convex hulls.  Insert a vertex labeled $\bar{a}$ in between $a$ and $a+1$ (where we identify $N+1$ and $1$).  When two convex hulls intersect in a vertex~$a$, there is a unique vertex~$\bar{b}$ that lies ``opposite'' to $a$ between the convex hulls intersecting in~$a$.  Connect $a$ and $\bar{b}$ by a line segment, which we call a \defn{diagonal}.  Two diagonals are \defn{adjacent} if they intersect a common convex hull.  Removing the convex hulls leaves only the diagonals, which by construction form a $(2k{+}2)$-angulation $\Theta(\mathbf{t})$ of a $2N$-gon.

  Conversely, any diagonal connects an even and an odd node in a $(2k{+}2)$-angulation of a convex $2N$-gon.  The convex hulls of the odd vertices in each $(2k{+}2)$-gon now give the factors in a commutation class of a factorization from $\Fact_k(c_N)$. 
\end{proof}

\begin{figure}
  \centering
  \begin{tikzpicture}\small
    \draw(1,1) node{\begin{tikzpicture}
        \polygon{(0,0)}{a}{16}{1.5}{.8}{$1$,$2$,$3$,$4$,$5$,$6$,$7$,$8$,$9$,$10$,$11$,$12$,$13$,$14$,$15$,$16$}[1.2];
        \begin{pgfonlayer}{background}
          \draw[fill=gray!50!white] (a4) to[bend right=30] (a5) to[bend left=10] (a15) to[bend right=30] (a16) to[bend right=30] (a4);
          \draw[fill=gray!50!white] (a1) to[bend right=30] (a2) to[bend right=30] (a3) to[bend left=30] (a16) to[bend right=30] (a1);
          \draw[fill=gray!50!white] (a10) to[bend right=30] (a11) to[bend right=30] (a12) to[bend right=30] (a13) to[bend left=30] (a10);
          \draw[fill=gray!50!white] (a5) to[bend right=30] (a6) to[bend right=10] (a13) to[bend right=30] (a14) to[bend right=10] (a5);
          \draw[fill=gray!50!white] (a6) to[bend right=30] (a7) to[bend right=30] (a8) to[bend right=30] (a9) to[bend left=30] (a6);
        \end{pgfonlayer}
      \end{tikzpicture}};
    \draw(6,1) node{\begin{tikzpicture}
        \polygonwhite{(0,0)}{a}{32}{1.5}{.8}{$ $,$ $,$ $,$ $,$ $,$ $,$ $,$ $,$ $,$ $,$ $,$ $,$ $,$ $,$ $,$ $,$ $,$ $,$ $,$ $,$ $,$ $,$ $,$ $,$ $,$ $,$ $,$ $,$ $,$ $,$ $,$ $}[1.2];
        \begin{pgfonlayer}{background}
          \draw[fill=black,thick](a31) to (a6);
          \draw[fill=black,thick](a9) to (a28);
          \draw[fill=black,thick](a11) to (a18);
          \draw[fill=black,thick](a25) to (a18);
        \end{pgfonlayer}
      \end{tikzpicture}};
    \draw(1,6.5) node{\begin{tikzpicture}
        \polygon{(0,0)}{a}{16}{1.5}{.8}{$1$,$2$,$3$,$4$,$5$,$6$,$7$,$8$,$9$,$10$,$11$,$12$,$13$,$14$,$15$,$16$}[1.2];
        \begin{pgfonlayer}{background}
          \draw[fill=gray!50!white] (a4) to[bend right=30] (a5) to[bend left=10] (a15) to[bend right=30] (a16) to[bend right=30] (a4);
          \draw[fill=gray!50!white] (a1) to[bend right=30] (a2) to[bend right=30] (a3) to[bend left=30] (a16) to[bend right=30] (a1);
          \draw[fill=gray!50!white] (a10) to[bend right=30] (a11) to[bend right=30] (a12) to[bend right=30] (a13) to[bend left=30] (a10);
          \draw[fill=gray!50!white] (a5) to[bend right=30] (a6) to[bend right=10] (a13) to[bend right=30] (a14) to[bend right=10] (a5);
          \draw[fill=gray!50!white] (a13) to[bend left=15] (a7) to[bend right=30] (a8) to[bend right=30] (a9) to[bend right=30] (a13);
        \end{pgfonlayer}
      \end{tikzpicture}};
    \draw(6,6.5) node{\begin{tikzpicture}
        \polygonwhite{(0,0)}{a}{32}{1.5}{.8}{$ $,$ $,$ $,$ $,$ $,$ $,$ $,$ $,$ $,$ $,$ $,$ $,$ $,$ $,$ $,$ $,$ $,$ $,$ $,$ $,$ $,$ $,$ $,$ $,$ $,$ $,$ $,$ $,$ $,$ $,$ $,$ $}[1.2];
        \begin{pgfonlayer}{background}
          \draw[fill=black,thick](a31) to (a6);
          \draw[fill=black,thick](a9) to (a28);
          \draw[fill=black,thick](a25) to (a12);
          \draw[fill=black,thick](a25) to (a18);
        \end{pgfonlayer}
      \end{tikzpicture}};
      \draw(3.6,1.2) node{$\Theta$};
      \draw[->](3.2,1) to (4,1);
      \draw(3.6,6.7) node{$\Theta$};
      \draw[->](3.2,6.5) to (4,6.5);
      \draw[->](1,3.2) to (1,4.2);
      \draw(-.7,3.7) node{inverse Hurwitz move}; 
      \draw[->](6,3.2) to (6,4.2);
      \draw(7.9,3.7) node{clockwise diagonal move};
  \end{tikzpicture}
  \caption{Illustration of the bijection $\Theta$ from ~\Cref{thm:1modk_factorizations_polygon_dissections} for $n=5$ and $k=3$.}
\end{figure}
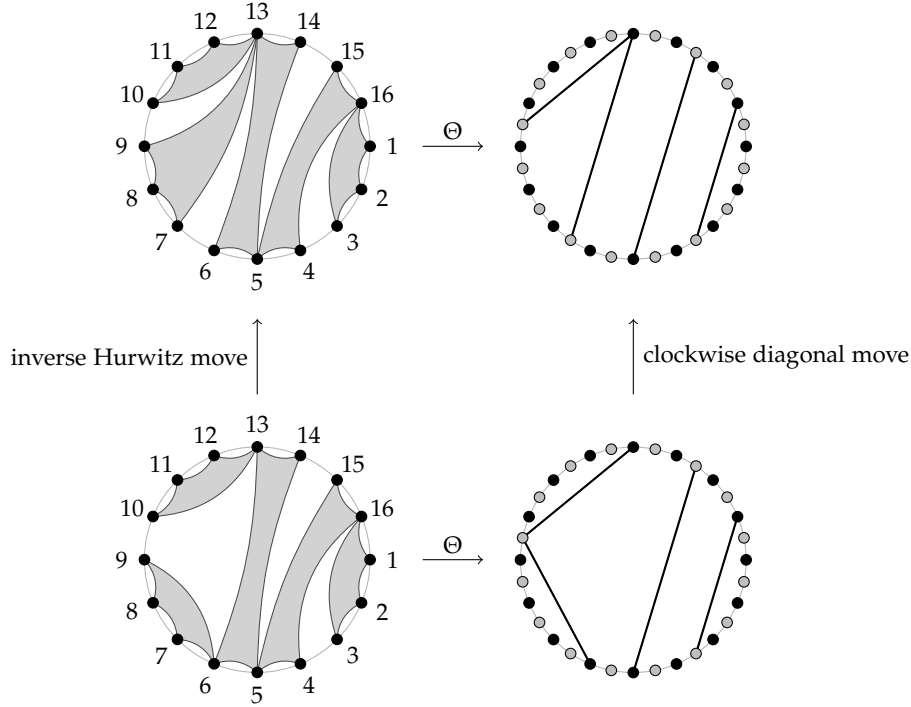

\begin{proposition}\label{prop:hurwitz_shifts}
  A Hurwitz move on a commutation-class of a reduced factorization corresponds to rotating a diagonal in the corresponding $(2k{+}2)$-angulation one step.
\end{proposition}
\begin{proof}
  Let $\mathbf{t}=t_1t_2\cdots t_n\in\Fact_k(c_N)$, and choose $i\in[n-1]$ such that $t_i$ and $t_{i+1}$ do not commute.  Then there is a unique integer $a$ which belongs to both $t_i$ and $t_{i+1}$.  Let $\bar{b}$ be the unique vertex in between the convex hulls of $t_i$ and $t_{i+1}$ visible from $a$.  Moreover, let $b_{i}$ denote the smallest element of $t_{i}$ greater or equal to $b+1$, and let $b_{i+1}$ denote the biggest element of $t_{i+1}$ less or equal to $b$.  

  Now, $\sigma_{i}\mathbb{t}$ is obtained by removing $a$ from $t_{i}$ and adding $b_{i+1}$ in the appropriate position (thus obtaining $t_{i+1}^{-1}t_it_{i+1}$), and by exchanging the order of these two factors.  Analogously, $\sigma_{i}^{-1}\mathbb{t}$ is obtained by removing $a$ from $t_{i+1}$ and adding $b_{i}$ in the appropriate position (thus obtaining $t_{i}t_{i+1}t_{i}^{-1}$), and by exchanging the order of these two factors.

  In view of the bijection $\Theta$ from \Cref{thm:1modk_factorizations_polygon_dissections} the $(2k{+}2)$-angulations $\Theta(\mathbf{t})$ and $\Theta(\sigma_i\mathbf{t})$ (respectively $\Theta(\sigma_i^{-1}\mathbf{t})$) differ by only shifting a unique diagonal.  More precisely, the diagonal connecting $a$ and $\bar{b}$ in $\Theta(\mathbf{t})$ is replaced by the diagonal connecting $b_{i+1}$ and $\bar{a}-1$ in $\Theta(\sigma_{i}\mathbf{t})$ (respectively by the diagonal connecting $b_{i}$ and $\bar{a}$ in $\Theta(\sigma_{i}^{-1}\mathbf{t})$).  Hence, the action of $\sigma_{i}$ (respectively $\sigma_{i}^{-1}$) corresponds to shifting a diagonal in counterclockwise (respectively clockwise) direction under $\Theta$.
\end{proof}

In \cite{stump18cataland}*{Section~6.6}, a lattice was constructed parametrized by a Coxeter group $W$, a Coxeter element $c\in W$, and an integer $m$; the \defn{$m$-Cambrian lattice} of $W$ with respect to the orientation $c$.  In the case where $W=\Symmetric_{n}$, and $c$ is given as the product of the simple transpositions in lexicographic order, the corresponding $m$-Cambrian lattice was realized combinatorially in \cite{freeze16combinatorial}*{Chapter~3} as a lattice on $(m{+}2)$-angulations of a convex $(mn{+}2)$-gon, where the cover relations are given by rotating a diagonal one step clockwise.  Let us refer to this lattice as the \defn{$(m,n)$-Cambrian lattice}.

\begin{corollary}\label{cor:2k_cambrian_lattice}
  The $(2k,n)$-Cambrian lattice is isomorphic to the poset whose elements are the reduced factorizations of $c_N$ up to commutation equivalence, with the cover relations given by Hurwitz moves.
\end{corollary}
\begin{proof}
  Consider the $(2kn+2)$-gon from the proof of~\Cref{thm:1modk_factorizations_polygon_dissections}, labeled clockwise by the numbers $1,\bar{1},2,\bar{2},\ldots,N,\bar{N}$.  We replace these labels as described in \cite{freeze16combinatorial}*{Section~3.2} starting from $\bar{N}$.  Under this substitution, the reduced factorization 
  \begin{displaymath}
    (1,2,\ldots,k{+}1)\cdot(k{+}1,k{+}2,\ldots,2k{+}1)\cdots(N{-}k,N{-}k{+}1,\ldots,N)
  \end{displaymath}
corresponds to the $(2k{+}2)$-angulation of the $(2kn+2)$-gon that is minimal in the $(2k,n)$-Cambrian lattice, and the reduced factorization 
  \begin{displaymath}
    (k{+}1,k{+}2,\ldots,2k{+}1)\cdot(2k{+}1,2k{+}2,\ldots,3k{+}1)\cdots(1,2,\ldots,k,N)
  \end{displaymath}
  corresponds to the $(2k{+}2)$-angulation of the $(2kn+2)$-gon that is maximal.  The claim then follows by~\Cref{thm:1modk_factorizations_polygon_dissections} and~\Cref{prop:hurwitz_shifts}.
\end{proof}

\Cref{fig:linear_2_cambrian} illustrates \Cref{cor:2k_cambrian_lattice} for $n=3$ and $k=1$.

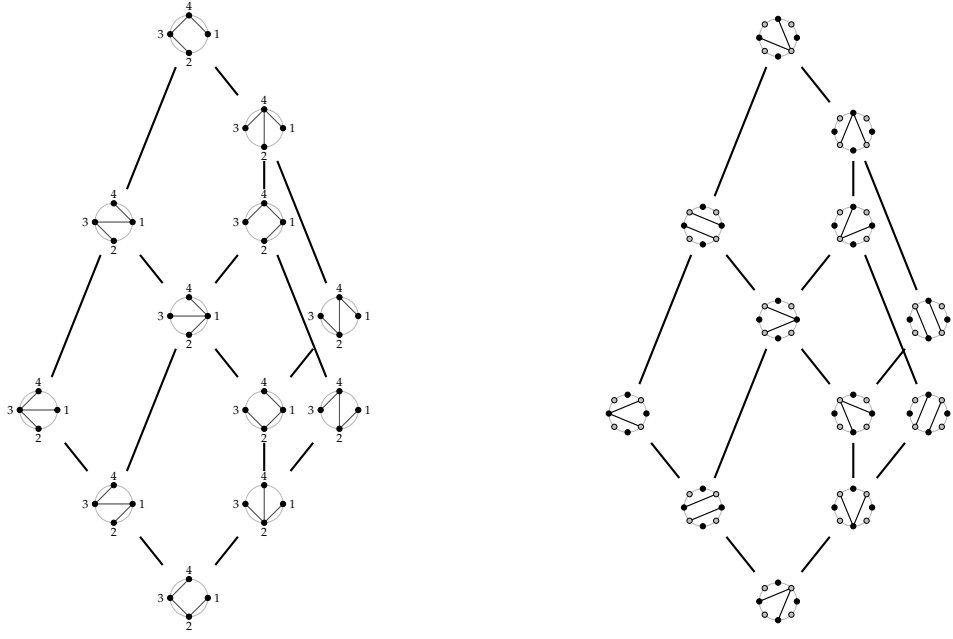
\begin{figure}
  \centering
  \begin{tikzpicture}\small
    \def\x{1};
    \def\y{1.25};
    \def\s{.5};
    \draw(3*\x,1*\y) node[scale=\s,inner sep=.2pt](n1){\ncFour{a1}{a2}{a2}{a3}{a3}{a4}};
    \draw(2*\x,2*\y) node[scale=\s,inner sep=.2pt](n2){\ncFour{a1}{a3}{a1}{a2}{a3}{a4}};
    \draw(4*\x,2*\y) node[scale=\s,inner sep=.2pt](n3){\ncFour{a1}{a2}{a2}{a3}{a2}{a4}};
    \draw(1*\x,3*\y) node[scale=\s,inner sep=.2pt](n4){\ncFour{a2}{a3}{a1}{a3}{a3}{a4}};
    \draw(4*\x,3*\y) node[scale=\s,inner sep=.2pt](n5){\ncFour{a1}{a4}{a1}{a2}{a2}{a3}};
    \draw(5*\x,3*\y) node[scale=\s,inner sep=.2pt](n6){\ncFour{a1}{a2}{a3}{a4}{a2}{a4}};
    \draw(3*\x,4*\y) node[scale=\s,inner sep=.2pt](n7){\ncFour{a1}{a4}{a1}{a3}{a1}{a2}};
    \draw(5*\x,4*\y) node[scale=\s,inner sep=.2pt](n8){\ncFour{a2}{a4}{a1}{a4}{a2}{a3}};
    \draw(2*\x,5*\y) node[scale=\s,inner sep=.2pt](n9){\ncFour{a1}{a4}{a2}{a3}{a1}{a3}};
    \draw(4*\x,5*\y) node[scale=\s,inner sep=.2pt](n10){\ncFour{a3}{a4}{a1}{a4}{a1}{a2}};
    \draw(4*\x,6*\y) node[scale=\s,inner sep=.2pt](n11){\ncFour{a3}{a4}{a2}{a4}{a1}{a4}};
    \draw(3*\x,7*\y) node[scale=\s,inner sep=.2pt](n12){\ncFour{a2}{a3}{a3}{a4}{a1}{a4}};
    \draw[thick](n1) -- (n2);
    \draw[thick](n1) -- (n3);
    \draw[thick](n2) -- (n4);
    \draw[thick](n2) -- (n7);
    \draw[thick](n3) -- (n5);
    \draw[thick](n3) -- (n6);
    \draw[thick](n4) -- (n9);
    \draw[thick](n5) -- (n7);
    \draw[thick](n5) -- (n8);
    \draw[thick](n6) -- (n10);
    \draw[thick](n7) -- (n9);
    \draw[thick](n7) -- (n10);
    \draw[thick](n8) -- (n11);
    \draw[thick](n9) -- (n12);
    \draw[thick](n10) -- (n11);
    \draw[thick](n11) -- (n12);
  \end{tikzpicture}
  \hfill
  \begin{tikzpicture}\small
    \def\x{1};
    \def\y{1.25};
    \def\s{.5};
    \draw(3*\x,1*\y) node[scale=\s,inner sep=.2pt](n1){\polyFour{a3}{a8}{a5}{a8}};
    \draw(2*\x,2*\y) node[scale=\s,inner sep=.2pt](n2){\polyFour{a1}{a4}{a5}{a8}};
    \draw(4*\x,2*\y) node[scale=\s,inner sep=.2pt](n3){\polyFour{a3}{a8}{a3}{a6}};
    \draw(1*\x,3*\y) node[scale=\s,inner sep=.2pt](n4){\polyFour{a5}{a2}{a5}{a8}};
    \draw(4*\x,3*\y) node[scale=\s,inner sep=.2pt](n5){\polyFour{a1}{a6}{a3}{a6}};
    \draw(5*\x,3*\y) node[scale=\s,inner sep=.2pt](n6){\polyFour{a3}{a8}{a7}{a4}};
    \draw(3*\x,4*\y) node[scale=\s,inner sep=.2pt](n7){\polyFour{a1}{a4}{a1}{a6}};
    \draw(5*\x,4*\y) node[scale=\s,inner sep=.2pt](n8){\polyFour{a3}{a6}{a7}{a2}};
    \draw(2*\x,5*\y) node[scale=\s,inner sep=.2pt](n9){\polyFour{a1}{a6}{a5}{a2}};
    \draw(4*\x,5*\y) node[scale=\s,inner sep=.2pt](n10){\polyFour{a1}{a4}{a7}{a4}};
    \draw(4*\x,6*\y) node[scale=\s,inner sep=.2pt](n11){\polyFour{a7}{a2}{a7}{a4}};
    \draw(3*\x,7*\y) node[scale=\s,inner sep=.2pt](n12){\polyFour{a5}{a2}{a7}{a2}};
    \draw[thick](n1) -- (n2);
    \draw[thick](n1) -- (n3);
    \draw[thick](n2) -- (n4);
    \draw[thick](n2) -- (n7);
    \draw[thick](n3) -- (n5);
    \draw[thick](n3) -- (n6);
    \draw[thick](n4) -- (n9);
    \draw[thick](n5) -- (n7);
    \draw[thick](n5) -- (n8);
    \draw[thick](n6) -- (n10);
    \draw[thick](n7) -- (n9);
    \draw[thick](n7) -- (n10);
    \draw[thick](n8) -- (n11);
    \draw[thick](n9) -- (n12);
    \draw[thick](n10) -- (n11);
    \draw[thick](n11) -- (n12);
  \end{tikzpicture}
  \caption{For $k=1$ and $n=3$, the $(2,3)$-Cambrian lattice realized as a lattice of reduced factorizations of $(1\;2\;3\;4)$ up to commutation equivalence (left), and realized as a lattice of quadrangulations of an $8$-gon (right).}
  \label{fig:linear_2_cambrian}
\end{figure}

\section{Nonnesting Partitions}
  \label{sec:1modk_nonnesting}
We also find analogues of the above construction in the world of nonnesting partitions.  Consider the \defn{triangular poset} defined by
\begin{displaymath}
	\Delta_{K} \defs \Bigl(\bigl\{(a,b)\mid 1\leq a<b\leq K\bigr\},\preceq\Bigr),
\end{displaymath}
where $(a,b)\preceq(c,d)$ if and only if $a\geq c$ and $b\leq d$.

We define $\Delta_{N;k}$ to be the induced subposet of $\Delta_{N-(k-1)}$ that consists of all pairs $(a,b)$ with $a\equiv 1\pmod{k}$. 
For $k=3$ and $n=4$ the poset $\Delta_{13;3}$ is shown in \Cref{fig:k_indivisible_nonnesting}.

We call an order ideal of $\Delta_{N;k}$ a \defn{$k$-indivisible nonnesting partition}, and we write $\nn_{N;k}$ for their set; for $k=1$ we get the usual nonnesting partitions. 
We may equivalently view $k$-indivisible nonnesting partitions as north-east paths from $(0,0)$ to $(N,n+1)$ that stay above the boundary path $\mathfrak{b}_{N,k}\defs UR(UR^{k})^n$.  Here we use the letter $U$ to indicate north-steps ($U$ for up), and the letter $R$ to indicate east-steps ($R$ for right).

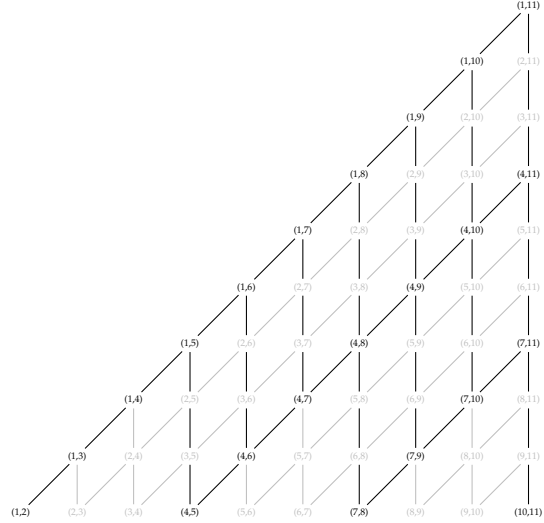
\begin{figure}
  \centering
  \begin{tikzpicture}\small
    \tikzmath{\x=.75;\y=.75;\s=.4;}
    \foreach \k in {1,...,10}{
      \tikzmath{int \dk; \dk = \k+1;}
      \foreach \l in {\dk,...,11}{
        \ifthenelse{\k=1 \OR \k=4 \OR \k=7 \OR \k=10}
          {\draw(\l*\x,{(\l-\k+1)*\y}) node[scale=\s](n\k\l){(\k,\l)};}
          {\draw[gray!50!white](\l*\x,{(\l-\k+1)*\y}) node[scale=\s](n\k\l){(\k,\l)};}
      }
    }
    \foreach \k in {1,...,9}{
      \tikzmath{int \dk; \dk = \k+1;}
      \foreach \l in {\dk,...,10}{
        \tikzmath{int \dl; \dl = \l+1;}
        \ifthenelse{\k=1 \OR \k=4 \OR \k=7}
          {\draw(n\k\l) -- (n\k\dl);}
          {\draw[gray!50!white](n\k\l) -- (n\k\dl);}
      }
    }
    \foreach \d in {1,2,3}{
      \draw[gray!50!white](3*\d*\x,2.15*\y) -- (3*\d*\x,2.85*\y);
      \draw[gray!50!white]({(3*\d+1)*\x},2.15*\y) -- ({(3*\d+1)*\x},2.85*\y);
      \draw[gray!50!white]({(3*\d+1)*\x},3.15*\y) -- ({(3*\d+1)*\x},3.85*\y);
    }
    \foreach \r in {0,1,2}{
      \tikzmath{int \dr; \dr=6-\r*3;}
      \foreach \d in {0,...,\dr}{
        \draw({(5+\r*3+\d)*\x},{(2.15+\d)*\y}) -- ({(5+\r*3+\d)*\x},{(2.85+\d)*\y});
        \draw({(5+\r*3+\d)*\x},{(3.15+\d)*\y}) -- ({(5+\r*3+\d)*\x},{(3.85+\d)*\y});
        \draw({(5+\r*3+\d)*\x},{(4.15+\d)*\y}) -- ({(5+\r*3+\d)*\x},{(4.85+\d)*\y});
      }
    }
  \end{tikzpicture}
  \caption{The poset $\Delta_{13;3}$ inside $\Delta_{11}$.   It has $340=\ran(4,4,2)$ order ideals.}
  \label{fig:k_indivisible_nonnesting}
\end{figure}

 Let $\mathcal{P}_{N;k}$ denote the set of all such paths. Recall that a \defn{$k$-Dyck path} of height $n$ is a north-east path from $(0,0)$ to $(kn,n)$ that stays weakly above the boundary path $(UR^k)^n$.  Let us write $\mathcal{D}_n^{(k)}$ to denote the set of all $k$-Dyck paths.  It follows from \cite{bizley54derivation} that the cardinality of $\mathcal{D}_n^{(k)}$ is the Fu{\ss}--Catalan number $\ran(n,k+1,1)$.

\begin{theorem}\label{thm:1modk_nonnesting_number}
  For $k,n\geq 1$, the set of order ideals of $\Delta_{N;k}$ is in bijection with the set of pairs of $k$-Dyck paths whose heights sum to $n$.  Consequently, we have $\bigl\lvert\nn_{N;k}\bigr\rvert=\ran(n,k+1,2)$.
\end{theorem}
\begin{proof}
	In terms of paths, this bijection is a standard decomposition that we detail here for completeness. For $\mathfrak{p}\in\mathcal{P}_{N;k}$ we say that $\mathfrak{p}$ \defn{touches} $\mathfrak{b}_{N;k}$ at step $i$, if the $i$-th east steps of $\mathfrak{p}$ and $\mathfrak{b}_{N;k}$ agree.  Every path in $\mathcal{P}_{N;k}$ touches $\mathfrak{b}_{N;k}$ at steps $N-k+1,N-k+2,\ldots,N$.

  Now let $\mathfrak{p}\in\mathcal{P}_{N;k}$ and fix the smallest $i\in\{0,1,\ldots,n\}$ such that $\mathfrak{p}$ touches $\mathfrak{b}_{N;k}$ at step $ik+1$.  We break $\mathfrak{p}$ in two pieces, by removing the first north-step and the $(ik{+}1)$-st east-step.  Let $\mathfrak{p}_{1}$ and $\mathfrak{p}_{2}$ denote the resulting paths.  Clearly, $\mathfrak{p}_{1}$ is a north-east path from $(0,1)$ to $(ik,i+1)$ that stays weakly above $R(UR^k)^{(i-1)}UR^{(k-1)}$, and $\mathfrak{p}_{2}$ is a north-east path from $(ik+1,i+1)$ to $(N,n+1)$ that stays weakly above $(UR^k)^{n-i}$.  Since $i$ was chosen minimal $\mathfrak{p}_{1}$ does not touch $\mathfrak{b}_{N;k}$ at $jk+1$ for $j<i$, which means that $\mathfrak{p}_{1}$ in fact stays above $(UR^{k})^{i}$.  Thus, $\mathfrak{p}_{1}\in\mathcal{D}_{i}^{(k)}$ and $\mathfrak{p}_{2}\in\mathcal{D}_{n-i}^{(k)}$.  We have just established
  \begin{align*}
    \bigl\lvert\mathcal{P}_{N;k}\bigr\rvert & = \sum_{i=0}^{n}{\bigl\lvert\mathcal{D}_{i}^{(k)}\bigr\rvert\cdot\bigl\lvert\mathcal{D}_{n-i}^{(k)}\bigr\rvert}\\
    & = \sum_{i=0}^{n}{\ran(i,k+1,1)\cdot\ran(n-i,k+1,1)}.
  \end{align*}
  Moreover, it is easily checked that for $n=1$ we have 
  \begin{displaymath}
    \bigl\lvert\mathcal{P}_{k+1;k}\bigr\rvert=2=\ran(1,k+1,2).
  \end{displaymath}
  By \Cref{lem:raney_recursion}, we find that the numbers $\bigl\lvert\mathcal{P}_{N;k}\bigr\rvert$ and $\ran(n,k+1,2)$ satisfy the same recurrence relation with the same initial conditions, and must therefore be equal.
\end{proof}

\begin{figure}
  \centering
  \begin{tikzpicture}\small
    \def\x{.3};
    \draw(3.5,5) node{\begin{tikzpicture}
      \draw[very thick,red](0*\x,0*\x) -- (0*\x,3*\x) -- (8*\x,3*\x) -- (8*\x,5*\x) -- (21*\x,5*\x) -- (21*\x,6*\x) -- (25*\x,6*\x) -- (25*\x,7*\x) -- (31*\x,7*\x);
      \draw[ultra thick,green!80!black](0*\x,0*\x) -- (0*\x,1*\x);
      \draw[ultra thick,green!80!black](20*\x,5*\x) -- (21*\x,5*\x);
      \begin{pgfonlayer}{background}
        \filldraw[fill=gray!50!white,draw=black,thick](0*\x,0*\x) -- (0*\x,1*\x) -- (1*\x,1*\x) -- (1*\x,2*\x) -- (6*\x,2*\x) -- (6*\x,3*\x) -- (11*\x,3*\x) -- (11*\x,4*\x) -- (16*\x,4*\x) -- (16*\x,5*\x) -- (21*\x,5*\x) -- (21*\x,6*\x) -- (26*\x,6*\x) -- (26*\x,7*\x) -- (31*\x,7*\x) -- (31*\x,0*\x) -- cycle;
      \end{pgfonlayer}
    \end{tikzpicture}};
    \draw(1,1) node{\begin{tikzpicture}
      \draw[very thick,red](0*\x,0*\x) -- (0*\x,2*\x) -- (8*\x,2*\x) -- (8*\x,4*\x) -- (20*\x,4*\x);
      \begin{pgfonlayer}{background}
        \filldraw[fill=gray!50!white,draw=black,thick](0*\x,0*\x) -- (1*\x,0*\x) -- (1*\x,1*\x) -- (6*\x,1*\x) -- (6*\x,2*\x) -- (11*\x,2*\x) -- (11*\x,3*\x) -- (16*\x,3*\x) -- (16*\x,4*\x) -- (20*\x,4*\x) -- (20*\x,0*\x) -- cycle;
        \foreach \k in {0,1,2,3}{
          \filldraw[fill=gray!30!white,draw=black,thick]({(0+5*\k)*\x},{(0+\k)*\x}) -- ({(1+5*\k)*\x},{(0+\k)*\x}) -- ({(1+5*\k)*\x},{(1+\k)*\x}) -- ({(0+5*\k)*\x},{(1+\k)*\x}) -- cycle;}
      \end{pgfonlayer}
    \end{tikzpicture}};
    \draw(7,1) node{\begin{tikzpicture}
      \draw(2*\x,4*\x) node{};
      \draw[very thick,red](0*\x,0*\x) -- (0*\x,1*\x) -- (4*\x,1*\x) -- (4*\x,2*\x) -- (10*\x,2*\x);
      \begin{pgfonlayer}{background}
        \filldraw[fill=gray!50!white,draw=black,thick](0*\x,0*\x) -- (0*\x,1*\x) -- (5*\x,1*\x) -- (5*\x,2*\x) -- (10*\x,2*\x) -- (10*\x,0*\x) -- cycle;
      \end{pgfonlayer}
    \end{tikzpicture}};
    \draw[->,very thick](2,3.5) -- (1.5,2);
    \draw[->,very thick](5.5,3.5) -- (6,2);
    \draw(0,6) node{$\mathfrak{p}\in\mathcal{P}_{31;5}$};
    \draw(-1,1.5) node{$\mathfrak{p}_{1}\in\mathcal{D}_{4}^{(5)}$};
    \draw(7.5,1.5) node{$\mathfrak{p}_{2}\in\mathcal{D}_{2}^{(5)}$};
  \end{tikzpicture}
  \caption{Illustration of the decomposition in the proof of \Cref{thm:1modk_nonnesting_number} for $n=6$ and $k=5$.  By construction, the path $\mathfrak{p}_{1}$ never enters the light-gray boxes.}
  \label{fig:1modk_nonnesting_decomposition}
\end{figure}
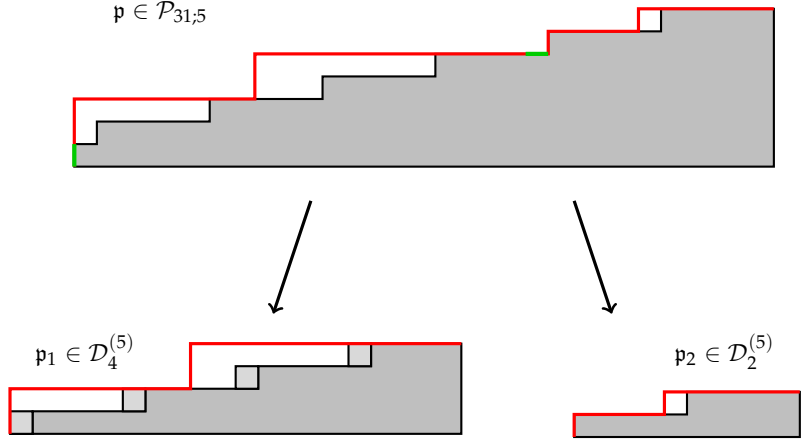

\Cref{fig:1modk_nonnesting_decomposition} illustrates the decomposition from the proof of ~\Cref{thm:1modk_nonnesting_number}.

\begin{corollary}
  For $n>2$ and $k\geq 1$ we have
  \begin{displaymath}
    \ran(n,k+1,2) = \sum_{i=1}^{n-1}{(-1)^{(i+1)}\binom{(n-i)k+2}{i}\ran(n-i,k+1,2)}.
  \end{displaymath}
\end{corollary}
\begin{proof}
	We have argued in \Cref{thm:1modk_nonnesting_number} that the order ideals of $\Delta_{N;k}$ are in bijection with north-east paths weakly above the boundary path $UR(UR^{k})^{n}$.  If we flip such a path together with the boundary path along the bottom border and rotate it by 90 degrees clockwise, we see that order ideals of $\Delta_{N;k}$ are in bijection with north-east paths weakly above $(U^{k}R)^{n}UR$.  Note that for such a path, the first $k$ steps must be north-steps, and the last step must be an east-step, so that we can forget these steps.  Consequently, the order ideals of $\Delta_{N;k}$ are in bijection with north-east paths weakly above $R(U^{k}R)^{n-1}U$.

	 In view of \cite{krattenthaler15lattice}*{Theorem~10.7.1} the number of such paths is given by the determinant of the matrix
  \begin{displaymath}
    M_{n;k} = \left(\binom{(n-j)k+2}{j-i+1}\right)_{1\leq i,j\leq n}.
  \end{displaymath}

  By Laplace expansion we see that for $n>2$ the determinant of $M_{n,k}$ satisfies the recursion given in the statement, and from~\Cref{thm:1modk_nonnesting_number} we conclude the result.
\end{proof}

\begin{remark}
  The set of $k$-Dyck paths of height $n$ is classically in bijection with the set of $(k{+}1)$-ary trees with $n$ non-leaf vertices. The bijections described in~\Cref{rem:1modk_bijection} and~\Cref{thm:1modk_nonnesting_number} thus extend to a bijection from $\nc_{N;k}$ to $\nn_{N;k}$.
\end{remark}

\section{Open Problems}
  \label{sec:1modk_mdivisible}

\subsection{EL-shellability}
  From a topological point of view, the lattice $\pnc_{N;1}$ of noncrossing partitions is particularly interesting: its order complex is a wedge of Catalan-many spheres.  This was established by Bj{\"o}rner and Edelman~\cite{bjorner80shellable}*{Remark~2} by showing that $\pnc_{N;1}$ admits a particular edge-labeling.  Such an \defn{EL-labeling} induces a shelling of the order complex, from which the mentioned property follows.

  We have attempted to extend this result to $\pnc_{N;k}$, but many natural choices for such a labeling did not have the desired properties.  Nevertheless, we still pose the following conjecture.

\begin{conjecture}
  The poset $\pnc_{N;k}$ admits an EL-labeling.  Consequently, the order complex of $\pnc_{N;k}$ with least and greatest elements removed is homotopic to a wedge of spheres.
\end{conjecture}

\subsection{Other types}
We give some conjectures for extending the combinatorics of this article to type $B$.  Fix simple reflections $s_0,s_1,\ldots,s_{kn-1}$ in the hyperoctahedral group of type $B_{kn}$ with $(s_0s_1)^4=1$. Analogously to the symmetric group, we group the transpositions of the factorization $c=s_0 s_1 \cdots s_{kn-1}$ of the linear Coxeter element as 
\[\mathbf{t}=(s_0\cdots s_{k-1})\cdot (s_{k}\cdots s_{2k-1})\cdots(s_{kn-k}\cdots s_{kn-1}).\]

\begin{conjecture}
The Hurwitz orbit of $\mathbf{t}$ contains $k^{n-1} n^n$ elements.
\end{conjecture}

We can take elements that occur as prefixes of the factorizations in the Hurwitz orbit of $\mathbf{t}$ to form the type $B_n$ $k$-indivisible noncrossing partitions.

\begin{conjecture}
There are $2 \binom{nk+n-1}{n-1}$ type $B_n$ $k$-indivisible noncrossing partitions. The zeta function of the restriction of the absolute order to those elements is $ q\binom{nk(q-1)+n-1}{n-1}$. 
\end{conjecture}

\section*{Acknowledgements}
We thank an anonymous reviewer for a careful reading of the manuscript and many insightful comments and suggestions.  We thank Christian Krattenthaler for providing a proof of \Cref{thm:1modk_rank_enumeration_mdivisible}, and suggesting to include it in this article.  N.W. thanks Jon McCammond for pointing him to~\cite{mccammond2017noncrossing}, Louis-Fran\c{c}ois Pr\'eville-Ratelle and Guillaume Chapuy for helpful conversations, and Christian Stump for providing TikZ code to draw noncrossing partitions. 

N.W. was partially supported by a Simons collaboration grant.


\end{document}